\documentclass[12pt]{amsart}
\usepackage{amssymb}
\usepackage{mathtools}
\usepackage{txfonts}
\usepackage{amsfonts}
\usepackage{eucal}
\usepackage{extarrows}

\usepackage{color}

\usepackage[all]{xy}

\usepackage{graphicx}
\usepackage{float}

\usepackage{xspace}

\usepackage[a4paper,body={16.3cm,22.8cm},centering]{geometry}

\ifpdf
\usepackage[colorlinks,final,backref=page,hyperindex]{hyperref}
\else
\usepackage[colorlinks,final,backref=page,hyperindex,hypertex]{hyperref}
\fi

\newcommand{\nc}{\newcommand}
\newcommand{\delete}[1]{}

\delete{
	\nc{\label}[1]{\label{#1}}  % Use this to suppress names
	\nc{\cite}[1]{\cite{#1}}  % Use this to suppress names
	\nc{\ref}[1]{\ref{#1}}  % Use this to suppress names
	\nc{\bibitem}[1]{\bibitem{#1}} % Use this to show number name
}

\newtheorem{theorem}{Theorem}[section]
\newtheorem{prop}[theorem]{Proposition}
\newtheorem{lemma}[theorem]{Lemma}
\newtheorem{coro}[theorem]{Corollary}

\theoremstyle{definition}
\newtheorem{defn}[theorem]{Definition}
\newtheorem{remark}[theorem]{Remark}
\newtheorem{exam}[theorem]{Example}
\newtheorem{prop-def}{Proposition-Definition}[section]

\newcommand\alphlist{a,b,c,d,e,f,g,h,i,j,k,l,m,n,o,p,q,r,s,t,u,v,w,x,y,z}
\newcommand\Alphlist{A,B,C,D,E,F,G,H,I,J,K,L,M,N,O,P,Q,R,S,T,U,V,W,X,Y,Z}
\newcommand\getcmds[3]{\expandafter\newcommand\csname #2#1\endcsname{#3{#1}}}
\makeatletter
\@for\x:=\alphlist\do{\expandafter\getcmds\expandafter{\x}{frak}{\mathfrak}}% \fraka,\frakb,...
\@for\x:=\Alphlist\do{\expandafter\getcmds\expandafter{\x}{frak}{\mathfrak}}% \frakA,\frakB,...
\makeatother

\nc{\bfk}{{\bf k}}
\font\cyr=wncyr10

\newfont{\scyr}{wncyr10 scaled 550}
\nc{\sha}{\mbox{\cyr X}}
\nc{\ssha}{\mbox{\bf \scyr X}}

\nc{\Id}{\mathrm{Id}}
\nc{\lbar}[1]{\overline{#1}}

\nc{\leaf}{\mathrm{leaf}}

	\nc{\fax}{\mathcal{F}(X)} %angularly decorated forests
	\nc{\pfa}{\mathcal{PF}}
	\nc{\shpr}{\diamond}    %Shuffle product
	\nc{\ot}{\otimes}      % tensor product
	\nc{\dep}{\mathrm{dep}} %depth of tree
	\nc{\var}{\varepsilon} %varepsilon
	\nc{\id}{\mathrm{id}}  %id
	\nc{\set}{\mathbf{Set}} % the cat of set
	\nc{\vect}{\mathbf{Vect}} % the cat of vector space
	\nc{\spep}{\mathbf{p}} %species p
	\nc{\speq}{\mathbf{q}} %species q
	\nc{\spei}{\mathbf{1}_{\bfk}} %species1
	\nc{\Hom}{\mathrm{Hom}}
	\nc{\calc}{\mathcal{C}}
	\nc{\adf}{\mathrm{ADF}}
	
	\nc{\dk}{\mathcal{K}}
	\nc{\wdk}{\widehat{\mathcal{K}}}
	\nc{\com}{\mathrm{Com}}
	\nc{\comp}{\mathrm{Comp}}

	\nc{\li}[1]{\textcolor{red}{#1}}
	\nc{\lir}[1]{\textcolor{red}{Li:#1}}
	\nc{\peng}[1]{\textcolor{purple}{Peng:#1}}
	\nc{\yi}[1]{\textcolor{cyan}{Yi:#1}}
	\nc{\revise}[1]{\textcolor{red}{#1}}

\begin{document}

\title{Multiple Rota-Baxter algebra and multiple Rota-Baxter modules}

\author{Jun He}
\address{School of Mathematics and Statistics,
	Nanjing University of Information Science \& Technology, Nanjing, Jiangsu 210044, P.\,R. China}
\email{he\_junn@163.com}

\author{Xiao-song Peng}
\address{School of Mathematics and Statistics,
	Jiangsu Normal University, Xuzhou, Jiangsu 221116, P.\,R. China}
\email{pengxiaosong3@163.com}

%\author{Yunzhou Xie}
%\address{Department of %Mathematics, Imperial College London, London SW7 2AZ, UK}
%\email{yx3021@ic.ac.uk}

\author{Yi Zhang}
\address{School of Mathematics and Statistics,
	Nanjing University of Information Science \& Technology, Nanjing, Jiangsu 210044, P.\,R. China}
\email{zhangy2016@nuist.edu.cn}

\begin{abstract}
	In this paper, we develop the theory of multiple Rota-Baxter modules over multiple Rota-Baxter algebras. We introduce left, right, and bimodule structures and construct free $\Omega$-operated modules with mixable tensor  establishing free commutative multiple Rota-Baxter modules. We provide a necessary and sufficient condition for a free module to admit a free multiple Rota-Baxter module structure. Furthermore, we define projective and injective multiple Rota-Baxter modules, showing that their category has enough projective and injective objects to support derived
	$\mathrm{Hom}$ functors. Finally, we introduce the tensor product of multiple Rota-Baxter algebras and define flat multiple Rota-Baxter modules, proving that both free and projective modules satisfy the flatness property.
	
\end{abstract}

\date{\today}
\subjclass[2020]{    %% 学科分类代码
	%16T30,	%Connections of Hopf algebras with combinatorics
	17B38,   %Yang-Baxter equations and Rota-Baxter operators,
	05E16,   %Combinatorial aspects of groups and algebras
	%16Z10, %Grobner-Shirshov bases
	%16T10, %Bialgebras
	%05C05,   %Trees
	16W99, %Rings and algebras with additional structure
	16S10, %Rings determined by universal properties (free algebras, coproducts, adjunction of inverses, etc.)
	%05A05   %words
}
\keywords{multiple Rota-Baxter module; multiple Rota-Baxter algebra; free modules; flat muitiple Rota-Baxter module }

\maketitle

\tableofcontents

\setcounter{section}{0}

\allowdisplaybreaks
%========================================================================
\section{Introduction}
\subsection{Rota-Baxter algebras and Rota-Baxter modules}
The notion of a Rota-Baxter operator was originally introduced by Baxter in 1960 \cite{B}. A Rota-Baxter algebra \cite{G2} refers to an associative algebra that is endowed with a Rota-Baxter operator, which satisfies  the integration-by-parts formula in analysis. Specifically, given a fixed parameter $\lambda$ in a commutative ring \( \mathbf{k} \), a Rota-Baxter algebra of weight \( \lambda \) is defined as a pair \( (R, P) \), where \( R \) is an associative algebra over \( \mathbf{k} \), and \( P: R \to R \) is a linear operator that satisfies the following Rota-Baxter identity:
\begin{equation}
	P(x)P(y) = P(xP(y)) + P(P(x)y) + \lambda P(xy), \quad \forall x, y \in R.
\end{equation}

From the 1960s to the 1990s, this structure attracted considerable attention from both analytic and combinatorial perspectives, with notable contributions from Atkinson, Cartier, and Rota \cite{G2}. In the setting of Lie algebras, the special case of a Rota-Baxter operator with weight 0 was independently discovered by physicists in relation to the operator form of the classical Yang-Baxter equation \cite{S}.

At the beginning of this century, the Rota-Baxter algebras, along with the Hopf algebra \cite{AGKO}, Lie algebras \cite{GK17}, dendriform algebra \cite{EG06},  pre-Lie algebras \cite{Agu00, AB, Bai22} and Post-Lie algebra \cite{BGN}, form the fundamental algebraic frameworks underlying the Connes-Kreimer renormalization theory \cite{CK} in quantum field theory. For further research developments, refer to Guo's monograph \cite{G2} and its references.
In recent years, the emergence of Rota-Baxter groups \cite{BG22, GLS21} has sparked renewed interest in the study of Rota-Baxter algebras. In particular, the discovery of new algebraic and geometric structures in mathematics and physics \cite{BGST24, BGLM, GGHZ, GLS21, LS23, STW24} has further advanced the research on Rota-Baxter Lie algebras and their applications in deformation theory and homotopy \cite{Das21, JSZ24, TBGS19}.

Similarly to classical algebraic structures including associative algebras and Lie algebras, the investigation of modules and representations constitutes an essential aspect of Rota-Baxter algebra theory. The concept of a Rota-Baxter module was initially introduced in \cite{GL} to study the representations of Rota-Baxter algebras and their connections with the ring structure of Rota-Baxter operators. To recall its definition, a {\bf Rota-Baxter module} is defined to be an $R$-module $M$ together with a linear operator $p$ on $M$ such that
\begin{equation}
	P(x)p(m)=p(xp(m))+p(P(x)m)+\lambda p(xm),\quad \forall x\in R,m\in M.
\end{equation}
Based on the studies, representation theory and regular-singular decomposition of Rota-Baxter algebraic structures over Laurent series were discussed \cite{LQ}. In \cite{QP}, representations of the polynomial Rota-Baxter algebras were investigated, and a classification of Rota-Baxter modules based on the indecomposable ${\bf k}[x]$-module was provided up to isomorphism. The geometric representations and derived functors of the Rota-Baxter modules from the perspective of the module category were studied in \cite{QGG}. Later, the concept of Rota-Baxter paired modules was introduced in \cite{ZGZ}, providing a natural generalization of Rota-Baxter modules while establishing their fundamental connections with Hopf algebra structures. The modules over free commutative non-unital Rota-Baxter algebra of nonzero weight were studied in \cite{TL}, and their equivalence to the modules over noncommutative algebra generated by two indeterminates with a generation relationship was established.

\subsection{Rota-Baxter algebras (modules) going multiple}

The structure of multiple compatible algebraic operations first appeared in Poisson and Lie brackets, originating from Magri's pioneering work on integrable Hamiltonian systems \cite{M78}.
Recently, a new algebraic structure involving multiple multilinear operators has emerged in the study of algebraic renormalization of regularity structures. In 2019, Bruned-Hairer-Zambotti \cite{BHZ19} used decorated rooted trees to characterize the renormalization process of stochastic partial differential equations. Subsequently, Foissy \cite{Foi21} introduced the concept of matching anticipative algebras to further understand the Hopf algebraic approach developed by Bruned, Hairer, and Zambotti for stochastic differential equations.

In recent years, the concept of matching Rota-Baxter algebras was introduced in \cite{GGZ}, which led to investigating matching Rota-Baxter modules and their relationships with modules over other algebraic structures in \cite{ZL}. Later, a generalization of Rota-Baxter algebras was proposed through the introduction of multiple Rota-Baxter algebras in \cite{ZZ}. Naturally, this development prompts the investigation of module structures over multiple Rota-Baxter algebras. In this paper, we focus on multiple Rota-Baxter modules and their categorical aspects.

{\bf Here is an outline of the paper.} In Section \ref{2}, we initially introduce the concept of left multiple Rota-Baxter modules and subsequently investigate right multiple Rota-Baxter modules and multiple Rota-Baxter bimodules through their module actions. Later, we construct free $\Omega$-operated modules with mixable tensor, and establish the construction of free commutative multiple Rota-Baxter modules over multiple Rota-Baxter algebra. Furthermore, we establish necessary and sufficient conditions under which a free module can be endowed with the structure of a multiple Rota-Baxter module. In Section \ref{3}, we define projective and injective multiple Rota-Baxter modules, and demonstrate that the category of multiple Rota-Baxter modules contains enough projective and injective objects. This enables the definition of derived functors of $\mathrm{Hom}$ through projective and injective resolutions of multiple Rota-Baxter modules. In Section \ref{4}, the tensor product is introduced over multiple Rota-Baxter algebras, providing the foundation for defining flat multiple Rota-Baxter modules. Furthermore, we show that both free multiple Rota-Baxter modules and, more generally, projective multiple Rota-Baxter modules satisfy the flatness property.

\smallskip
{\bf Notation.}
Let $\bf k$ be a commutative unital ring. All algebras (modules) discussed in this paper are assumed to be $\bf k$-algebras ($\bf k$-modules) with unity, and all morphisms are
$\bf k$-linear and preserve the unity element unless stated otherwise.

\section{Free multiple Rota-Baxter modules}\label{2}
In this section, we introduce the concept of multiple Rota-Baxter modules. Then we provide a construction of the free multiple Rota-Baxter modules by $\Omega$-operated modules.

\subsection{Multiple Rota-Baxter modules}
We first recall the notion of a multiple Rota-Baxter algebra and a Rota-Baxter module from \cite{GGZ,QGG,ZL}.
\begin{defn}
	Let $\Omega$ be a nonempty set and let $\lambda_{\Omega}:=(\lambda_{\omega})_{\omega \in \Omega},\lambda_{\omega}\in\bf k $ be an indexed set of scalars indexed by $\Omega$.
	\begin{enumerate}
		\item A {\bf multiple Rota-Baxter algebra} $(R,P_{\Omega})$ of pair weight $(\lambda_{\Omega},\lambda_{\Omega})$ is a pair $(R, P_{\Omega})$, where $R$ is an  algebra, $P_{\Omega}:=(P_{\omega})_{\omega \in \Omega }$ is a set of linear operators $P_\omega:R\longrightarrow R$, which satisfy the multiple Rota-Baxter identity
		\begin{equation}{\label{a_0}}
			\begin{split}
				P_{\alpha}(r_1)P_{\beta}(r_2)&=P_{\alpha}(r_1P_{\beta}(r_2))+P_{\beta}(P_{\alpha}(r_1)r_2)+\lambda_{\beta}P_{\alpha}(r_1r_2)+\lambda_{\alpha}P_{\beta}(r_1r_2)\\
				&=P_{\alpha}(r_1P_{\beta}(r_2)+\lambda_{\beta}(r_1r_2))+P_{\beta}(P_{\alpha}(r_1)r_2+\lambda_{\alpha}(r_1r_2)),
			\end{split}
		\end{equation}
		for all $r_1,r_2\in R,\alpha,\beta \in \Omega$.
		\item A {\bf left $(R,P_{\Omega})$-ideal} $(I,I_{\Omega})$ is an ideal $I$ of $R$ equipped with a set of linear operators $I_{\Omega}:=P_{\Omega}|_I$ such that $P_{\omega}(I)\subseteq I$ for all $\omega\in \Omega$.
	\end{enumerate}
\end{defn}
\begin{remark}
	\begin{enumerate}
		\item
		Any Rota-Baxter algebra of weight \(\lambda\) can be interpreted as a multiple Rota-Baxter algebra of pair weight \((\lambda,0)\) or \((0,\lambda)\) by selecting \(\Omega\) as a single-element set.
		\item Any matching Rota-Baxter of weight $\lambda_\Omega$ is a multiple Rota-Baxter algebra of pair weight $(\lambda_\Omega,0)$.
	\end{enumerate}
\end{remark}

In what follows, unless otherwise specified, we suppose that $\Omega$ is a nonempty set, $\lambda_{\Omega}:=(\lambda_{\omega})_{\omega \in \Omega},\lambda_{\omega}\in \bf k $ is a set of scalars indexed by $\Omega$ and $(R,P_{\Omega})$ is a multiple Rota-Baxter algebra of pair weight $(\lambda_{\Omega},\lambda_{\Omega})$.

\begin{defn}
	With the above notations.
	\begin{enumerate}
		\item A {\bf left $(R, P_{\Omega})$-module} \((M, \mathfrak{m}_{\Omega})\) is a left \(R\)-module \(M\) equipped with a set $\mathfrak{m}_{\Omega}:=(\mathfrak{m}_{\omega})_{\omega \in \Omega}$ of linear operators \(\mathfrak{m}_{\omega}: M \to M\) satisfying the identity
		\begin{equation}{\label{a}}
			P_{\alpha}(x)\mathfrak{m}_{\beta}(m) = \mathfrak{m}_{\alpha}(x\mathfrak{m}_{\beta}(m)) + \mathfrak{m}_{\beta}(P_{\alpha}(x)m) + \lambda_{\beta} \mathfrak{m}_{\alpha}(xm) + \lambda_{\alpha} \mathfrak{m}_{\beta}(xm),
		\end{equation}
		for all \(x \in R\), \(m \in M\), and \(\alpha, \beta \in \Omega\).
		
		\item Let $(M,\mathfrak{m}_{\Omega})$ and $(M^{'},\mathfrak{m}'_{\Omega})$ be left $(R,P_{\Omega})$-modules. A {\bf left $(R,P_{\Omega})$-module homomorphism} $\phi:(M,\mathfrak{m}_{\Omega})\longrightarrow (M',\mathfrak{m}_{\Omega}')$ is a left $R$-module homomorphism $\phi:M\longrightarrow M^{'}$ such that $\phi\circ \mathfrak{m}_{\omega}=\mathfrak{m}_{\omega}^{'}\circ\phi$ for all $\omega \in \Omega$.
		
		\item Let $(M,\mathfrak{m}_{\Omega})$ be a left $(R,P_{\Omega})$-module. A {\bf left $(R,P_{\Omega})$-submodule} of $(M,\mathfrak{m}_{\Omega})$ is a pair $(N,\mathfrak{n}_{\Omega})$ consists of a submodule $N$ of $M$ and a set of linear operators $\mathfrak{n}_{\Omega}:=\mathfrak{m}_{\Omega}|_N$ such that $\mathfrak{m}_{\omega}(N)\subseteq N$ for all $\omega\in\Omega$.
	\end{enumerate}
	
\end{defn}

\begin{lemma}{\label{A}}
	Let $(M,\mathfrak{m}_{\Omega})$ be a left $(R,P_{\Omega})$-module and $(N,\mathfrak{n}_{\Omega})$ a left $(R,P_{\Omega})$ submodule of $(M,\mathfrak{m}_{\Omega})$. For each $\omega\in\Omega$, define
	\begin{equation*}
		\overline{\mathfrak{m}}_{\omega}: M/N\longrightarrow M/N,\quad \overline{\mathfrak{m}}_{\omega}(m+N)= \mathfrak{m}_{\omega}(m)+N,\quad \omega\in\Omega.
	\end{equation*}
	Then $(M/N,\overline{\mathfrak{m}}_{\Omega})$ is a left $(R,P_{\Omega})$-module with $\overline{\mathfrak{m}}_{\Omega}:=(\overline{\mathfrak{m}}_{\omega})_{\omega\in\Omega}$. We refer to $(M/N,\overline{\mathfrak{m}}_{\Omega})$ as the {\bf quotient $(R,P_{\Omega})$-module} of $(M,\mathfrak{m}_{\Omega})$ by $(N,\mathfrak{n}_{\Omega})$.
\end{lemma}

\begin{proof}
	The definition of $\overline{\mathfrak{m}}_{\Omega}$ is well-defined by $\mathfrak{m}_{\omega}(N)\subseteq N$. Then we need to prove that $\overline{\mathfrak{m}}_{\omega}$ satisfies Eq.~(\ref{a}). For $r\in R,m\in M,\alpha,\beta\in \Omega$,
	\begin{equation*}
		\begin{split}
			P_{\alpha}(r)\overline{\mathfrak{m}}_{\beta}(m+N)=& P_{\alpha}(r)(\mathfrak{m}_{\beta}(m)+N)\\
			=& P_{\alpha}(r)\mathfrak{m}_{\beta}(m)+N\\
			=& \mathfrak{m}_{\alpha}(r\mathfrak{m}_{\beta}(m))+\mathfrak{m}_{\beta}(P_{\alpha}(r)m)
			+\lambda_{\alpha}\mathfrak{m}_{\beta}(rm)+\lambda_{\beta}\mathfrak{m}_{\beta}(rm)+N\\
			=& (\mathfrak{m}_{\alpha}(r\mathfrak{m}_{\beta}(m))+N)+(\mathfrak{m}_{\beta}(P_{\alpha}(r)m)+N)\\
			&+(\lambda_{\alpha}\mathfrak{m}_{\beta}(rm)+N)+(\lambda_{\beta}\mathfrak{m}_{\beta}(rm)+N)\\
			=& \overline{\mathfrak{m}}_{\alpha}(r\overline{\mathfrak{m}}_{\beta}(m+N))+\overline{\mathfrak{m}}_{\beta}(P_{\alpha}(r)(m+N))\\
			&+\lambda_{\alpha}\overline{\mathfrak{m}}_{\beta}(r(m+N))+\lambda_{\beta}\overline{\mathfrak{m}}_{\alpha}(r(m+N)),
		\end{split}
	\end{equation*}
	as required.
	% So $(M/N,\overline{\mathfrak{m}}_{\Omega})$ is a left $(R,P_{\Omega})$-module.
\end{proof}

\begin{exam}
	\begin{enumerate}
		\item Consider the left action of $(R,P_{\Omega})$ on itself,  $(R,P_{\Omega})$ is also a left $(R,P_{\Omega})$-module.
		
		\item Left $(R,P_{\Omega})$-ideal $I$ is a special case of left $(R,P_{\Omega})$-module with restriction $P_{\omega}|_I: I\longrightarrow I$. Then $(R/I,\overline{P}_{\Omega})$ is also a left $(R,P_{\Omega})$-module with $\overline{P}_{\omega}:R/I\longrightarrow R/I$ by Lemma ~\ref{A}.
	\end{enumerate}
\end{exam}

\begin{prop}
	Let $(M, \mathfrak{m}_{\Omega})$ be a left $(R, P_{\Omega})$-module. We consider a set of maps $A_i : \Omega \longrightarrow {\bf k}$ with finite support, indexed by $i \in I$, identified with $A_i= \{a_{i,\omega} \in {\bf k} | \omega \in \Omega\}$. Define the linear combinations:
	\begin{equation*}
		P_i := P_{A_i} := \sum_{\omega \in \Omega} a_{i,\omega} P_{\omega}, \quad
		\overline{\mathfrak{m}}_i := \overline{\mathfrak{m}}_{A_i} := \sum_{\omega \in \Omega} a_{i,\omega} \mathfrak{m}_{\omega}.
	\end{equation*}
	
	Then:
	\begin{enumerate}
		\item $(R, P_I)$ is a multiple Rota-Baxter algebra of pair weight $(\lambda_I, \lambda_I)$, where:
		\begin{align*}
			P_I := \{P_i \mid i \in I\},\quad
			\lambda_I := \left\{\lambda_i := \sum_{\omega \in \Omega} a_{i,\omega} \lambda_{\omega} \,\bigg|\, i \in I\right\}.
		\end{align*}
		
		\item $(M, \overline{\mathfrak{m}}_I)$ is a left $(R, P_I)$-module of pair weight $(\lambda_I, \lambda_I)$ with $\overline{\mathfrak{m}}_I := \{\overline{\mathfrak{m}}_i \mid i \in I\}$.
	\end{enumerate}
	
	This shows that within any left $(R, P_{\Omega})$-module, linear combinations of the Rota-Baxter operators $P_{\omega}$ and $\mathfrak{m}_{\omega}$ preserve the Rota-Baxter module structure with appropriately adjusted weights.
\end{prop}
\begin{proof}
	(a). By \cite{ZZ}, $(R,P_I)$ is a multiple Rota-Baxter algebra of pair weight $(\lambda_I,\lambda_I)$.

	(b). We need to verify $\overline{\mathfrak{m}}_{\omega}$ satisfying Eq.~(\ref{a}). For $x \in R$, $m\in M$, $\alpha,\beta\in\Omega$ and $i,j \in I$, we have
	\begin{align*}
		P_i(x)\overline{\mathfrak{m}}_j(m)=&\left(\sum_{\alpha\in\Omega}^{}a_{i,\alpha}P_{\alpha}(x)\right)\left(\sum_{\beta\in\Omega}^{}a_{j,\beta}\mathfrak{m}_{\beta}(m)\right)\\
		=&\sum_{\alpha\in\Omega}^{}\sum_{\beta\in\Omega}^{}a_{i,\alpha}a_{j,\beta}P_{\alpha}(x)\mathfrak{m}_{\beta}(m)\\
		=&\sum_{\alpha\in\Omega}^{}\sum_{\beta\in\Omega}^{}a_{i,\alpha}a_{j,\beta}\left(\mathfrak{m}_{\alpha}(x\mathfrak{m}_{\beta}(m))+\mathfrak{m}_{\beta}(P_{\alpha}(x)m)+\lambda_{\beta}\mathfrak{m}_{\alpha}(xy)+\lambda_{\alpha}\mathfrak{m}_{\beta}(xm)\right)\\
		=&\sum_{\alpha\in\Omega}^{}\sum_{\beta\in\Omega}^{}a_{i,\alpha}a_{j,\beta}\mathfrak{m}_{\alpha}(x\mathfrak{m}_{\beta}(m))+\sum_{\alpha\in\Omega}^{}\sum_{\beta\in\Omega}^{}a_{i,\alpha}a_{j,\beta}\mathfrak{m}_{\beta}(P_{\alpha}(x)m)\\
		&+\sum_{\alpha\in\Omega}^{}\sum_{\beta\in\Omega}^{}a_{i,\alpha}a_{j,\beta}\lambda_{\beta}\mathfrak{m}_{\alpha}(xm)+\sum_{\alpha\in\Omega}^{}\sum_{\beta\in\Omega}^{}a_{i,\alpha}a_{j,\beta}\lambda_{\alpha}\mathfrak{m}_{\beta}(xm)\\
		=&\sum_{\alpha\in\Omega}^{}a_{i,\alpha}\mathfrak{m}_{\alpha}\left(x\sum_{\beta\in\Omega}^{}a_{j,\beta}\mathfrak{m}_{\beta}(m)\right)+\sum_{\beta\in\Omega}^{}a_{j,\beta}\mathfrak{m}_{\beta}\left(\sum_{\alpha\in\Omega}^{}a_{i,\alpha}P_{\alpha}(x)m\right)\\
		&+\sum_{\beta\in\Omega}^{}a_{j,\beta}\lambda_{\beta}\left(\sum_{\alpha\in\Omega}^{}a_{i,\alpha}\mathfrak{m}_{\alpha}(xm)\right)+\sum_{\alpha\in\Omega}^{}a_{i,\alpha}\lambda_{\alpha}\left(\sum_{\beta\in\Omega}^{}a_{j,\beta}\mathfrak{m}_{\beta}(xm)\right)\\
		=&\sum_{\alpha\in\Omega}^{}a_{i,\alpha}\mathfrak{m}_{\alpha}(x\overline{\mathfrak{m}}_j(m))+\sum_{\beta\in\Omega}^{}a_{j,\beta}\mathfrak{m}_{\beta}(P_i(x)m)\\
		&+\sum_{\beta\in\Omega}^{}a_{j,\beta}\lambda_{\beta}\overline{\mathfrak{m}}_i(xm)+\sum_{\alpha\in\Omega}^{}a_{i,\alpha}\lambda_{\alpha}\overline{\mathfrak{m}}_j(xm)\\
		=&\overline{\mathfrak{m}}_i(x\overline{\mathfrak{m}}_j(m))+\overline{\mathfrak{m}}_j(P_i(x)m)+\sum_{\beta\in\Omega}^{}a_{j,\beta}\lambda_{\beta}\overline{\mathfrak{m}}_i(xm)+\sum_{\alpha\in\Omega}^{}a_{i,\alpha}\lambda_{\alpha}\overline{\mathfrak{m}}_j(xm)\\
		=&\overline{\mathfrak{m}}_i(x\overline{\mathfrak{m}}_j(m))+\overline{\mathfrak{m}}_j(P_i(x)m)+\lambda_j\overline{\mathfrak{m}}_i(xm)+\lambda_i\overline{\mathfrak{m}}_j(xm),
	\end{align*}
	as required.
\end{proof}

By considering left $R$-multiplication and the action ${\mathfrak{m}}_{\omega}$ as linear operators on $\mathrm{End}(M)$, then Eq.~(\ref{a})
% :
% \begin{equation*}
	% 	P_{\alpha}(x)\mathfrak{m}_{\beta}(m)=\mathfrak{m}_{\alpha}(x\mathfrak{m}_{\beta}(m))+\mathfrak{m}_{\beta}(P_{\alpha}(x)m)+\lambda_{\beta}\mathfrak{m}_{\alpha}(xm)+\lambda_{\alpha}\mathfrak{m}_{\beta}(xm),
	% \end{equation*}
can be written as
\begin{equation*}
	(P_{\alpha}(x)\circ {\mathfrak{m}}_\beta)(m)=(\mathfrak{m}_{\alpha}\circ x\circ \mathfrak{m}_{\beta})(m)+(\mathfrak{m}_{\beta}\circ P_{\alpha}(x))(m)+\lambda_{\beta}(\mathfrak{m}_{\alpha}\circ x)(m)+\lambda_{\alpha}(\mathfrak{m}_{\beta}\circ x)(m).
\end{equation*}
A left module over an algebra can be regarded as a right module over the opposite algebra. Then considering the right $\mathrm{End}(M)$-action to $M$, we obtain
\begin{equation*}
	(m)(P_{\alpha}(x)\circ \mathfrak{m}_{\beta})=(m)(\mathfrak{m}_{\alpha}\circ x\circ \mathfrak{m}_{\beta})+(m)(\mathfrak{m}_{\beta}\circ P_{\alpha}(x))+\lambda_{\beta}(m)(\mathfrak{m}_{\alpha}\circ x)+\lambda_{\alpha}(m)(\mathfrak{m}_{\beta}\circ x),
\end{equation*}
which can be written as
\begin{equation*}
	(mP_{\alpha}(x))\mathfrak{m}_{\beta}=(((m)\mathfrak{m}_{\alpha})x)\mathfrak{m}_{\beta}+((m)\mathfrak{m}_{\beta})P_{\alpha}(x)+\lambda_{\beta}((m)\mathfrak{m}_{\alpha})x+\lambda_{\alpha}((m)\mathfrak{m}_{\beta})x.
\end{equation*}
Hence we give the definition of right $(R,P_{\Omega})$-module.

\begin{defn}
	A {\bf right $(R,P_{\Omega})$-module} $(M,\mathfrak{m}_{\Omega})$ is a pair $(M,\mathfrak{m}_{\Omega})$ consists of a right $R$-module $M$ and a set $\mathfrak{m}_{\Omega}:=(\mathfrak{m}_{\omega})_{\omega \in \Omega}$ of linear operators $\mathfrak{m}_{\omega}:M\longrightarrow M$ such that
	\begin{equation}{\label{b}}
		\mathfrak{m}_{\beta}(mP_{\alpha}(x))=\mathfrak{m}_{\beta}(\mathfrak{m}_{\alpha}(m)x)+\mathfrak{m}_{\beta}(m)P_{\alpha}(x)+\lambda_{\beta}\mathfrak{m}_{\alpha}(m)x+\lambda_{\alpha}\mathfrak{m}_{\beta}(m)x,
	\end{equation}
	for all $x\in R,m\in M,\alpha,\beta \in \Omega$.
\end{defn}

\begin{defn}
	Let $(R,P_{\Omega})$ and $(R',P_{\Omega}')$ be multiple Rota-Baxter algebras of pair weight $(\lambda_{\Omega},\lambda_{\Omega})$, an $(R,P_{\Omega})$-$(R',P_{\Omega}')$-{\bf bimodule} is a triple $(_{(R,P_{\Omega})}M_{(R',P_{\Omega}')},\mathfrak{m}_{\Omega}^R,\mathfrak{m}_{\Omega}^{R'})$, where
	\begin{enumerate}
		\item $(M,\mathfrak{m}_{\Omega}^R)$ is a left $(R,P_{\Omega})$-module,
		\item $(M,\mathfrak{m}_{\Omega}^{R'})$ is a right $(R',P_{\Omega}')$-module,
		\item $M$ carries an $R$-$R'$-bimodule structure,
	\end{enumerate}
	satisfying the compatibility conditions
	\begin{equation*}
		\mathfrak{m}_{\Omega}^{R'}(rm)=r\mathfrak{m}_{\Omega}^{R'}(m),\quad \mathfrak{m}_{\Omega}^R(mr')=\mathfrak{m}_{\Omega}^R(m)r',\quad \mathfrak{m}_{\Omega}^{R'}(\mathfrak{m}_{\Omega}^R(m))=\mathfrak{m}_{\Omega}^R(\mathfrak{m}_{\Omega}^{R'}(m)),
	\end{equation*}
	for all $m\in M,r\in R,r'\in {R'}$.
\end{defn}

\subsection{Free multiple Rota-Baxter modules}
Recall \cite{G2,G1} that, an \textbf{$\Omega$-operated algebra} is defined as an algebra $R$ equipped with a set $\Omega:=(\omega)_{\omega\in\Omega}$ of linear operators $\omega: R \to R$. In this subsection, we give the concept of $\Omega$-operated modules and free left $(R,P_{\Omega})$ modules and give a construction of free left $(R,P_{\Omega})$-modules.
\begin{defn}
	Let $(R,\Omega)$ be an $\Omega$-operated algebra.
	\begin{enumerate}
		\item A {\bf left $\Omega$-operated module} $(M,(\mathfrak{m}_{\omega})_{\omega\in\Omega})$ is a left $R$-module $M$ with a set of linear operators $\mathfrak{m}_{\omega}:M\longrightarrow M$.
		
		\item Let $(M,(\mathfrak{m}_{\omega})_{\omega\in\Omega})$ and $(N,(\mathfrak{n}_{\omega})_{\omega\in\Omega})$ be left $\Omega$-operated modules. A {\bf left $\Omega$-operated module homomorphism }$\phi:(M,(\mathfrak{m}_{\omega})_{\omega\in\Omega})\longrightarrow (N,(\mathfrak{n}_{\omega})_{\omega\in\Omega})$ is a left $R$-module homomorphism $\phi:M\longrightarrow N$ such that $\phi\circ \mathfrak{m}_{\omega}=\mathfrak{n}_{\omega}\circ \phi$ for all $\omega\in\Omega$.
		
		\item A {\bf free left $\Omega$-operated module on a set $X$} is a left $\Omega$-operated module $(M,(\mathfrak{m}_{\omega})_{\omega\in\Omega})$ together with any set map $j_{X}:X\longrightarrow M$ with the universal property that, for any left $\Omega$-operated module $(N,(\mathfrak{n}_{\omega})_{\omega\in\Omega})$ together with any set map $f:X\longrightarrow N$, there is a unique left $\Omega$-operated module homomorphism $\overline{f}:(M,(\mathfrak{m}_{\omega})_{\omega\in\Omega})\longrightarrow (N,(\mathfrak{n}_{\omega})_{\omega\in\Omega})$ such that $\overline{f}\circ j_X=f$.
	\end{enumerate}
\end{defn}
Apparently, left multiple Rota-Baxter modules are left $\Omega$-operated modules.

Next we construct the free left $\Omega$-operated module on a set $X$.

Suppose $(R,\Omega)$ is an $\Omega$-operated algebra and $X$ is a set. We consider the following sequences
\begin{equation*}
	\begin{split}
		&M_1(X):=R\otimes {\bf k}X,\\
		&M_2(X):=M_1(X)\oplus R\otimes {\bf k}\Omega\otimes M_1(X),\\
		&\quad \vdots\\
		&M_n(X):=M_{n-1}(X)\oplus R\otimes {\bf k}\Omega\otimes M_{n-1}(X),\\
		&\quad \vdots
	\end{split}
\end{equation*}
denote $M_R(X):=\cup^{\infty}_{n=1}M_n(X)$. For pure tensor $r_1\otimes \omega_1\otimes\cdots\otimes r_{n-1}\otimes\omega_{n-1}\otimes r_n\otimes x\in M_R(X)$ with $r_i\in R$, $\omega_i\in\Omega$, $x\in X$, $r\in R$, define the action of $R$ on $M_R(X)$ as
\begin{equation*}
	r(r_1\otimes \omega_1\otimes\cdots\otimes r_{n-1}\otimes\omega_{n-1}\otimes r_n\otimes x):= rr_1\otimes \omega_1\otimes\cdots\otimes r_{n-1}\otimes\omega_{n-1}\otimes r_n\otimes x.
\end{equation*}

\begin{prop}
	Let $(R,\Omega)$ be an $\Omega$-operated algebra and $X$ a set. Then $M_{R}(X)$ carries a left $R$-module structure with above notions.
\end{prop}
\begin{proof}
	We need to verify $M_R(X)$ satisfies the axioms of module.
	First, for $r_1\otimes \omega_1\otimes\cdots\otimes r_{s-1}\otimes\omega_{s-1}\otimes r_s\otimes x,\  r_1'\otimes \omega_1'\otimes\cdots\otimes r_{t-1}'\otimes\omega_{t-1}'\otimes r_t'\otimes x\in M_R(X)$ and $r,r'\in R$, we have
	\begin{align*}
		&r(r_1\otimes \omega_1\otimes\cdots\otimes r_{s-1}\otimes\omega_{s-1}\otimes r_s\otimes x+ r_1'\otimes \omega_1'\otimes\cdots\otimes r_{t-1}'\otimes\omega_{t-1}'\otimes r_t'\otimes x)\\
		=&rr_1\otimes \omega_1\otimes\cdots\otimes r_{s-1}\otimes\omega_{s-1}\otimes r_s\otimes x+rr_1'\otimes \omega_1'\otimes\cdots\otimes r_{t-1}'\otimes\omega_{t-1}'\otimes r_t'\otimes x\\
		=&r(r_1\otimes \omega_1\otimes\cdots\otimes r_{s-1}\otimes\omega_{s-1}\otimes r_s\otimes x)+r(r_1'\otimes \omega_1'\otimes\cdots\otimes r_{t-1}'\otimes\omega_{t-1}'\otimes r_t'\otimes x),
	\end{align*}
and
	\begin{align*}
		&(r+r')(r_1\otimes \omega_1\otimes\cdots\otimes r_{s-1}\otimes\omega_{s-1}\otimes r_s\otimes x)\\
		=&(r+r')\otimes r_1\otimes \omega_1\otimes\cdots\otimes r_{s-1}\otimes\omega_{s-1}\otimes r_s\otimes x\\
		=&rr_1\otimes \omega_1\otimes\cdots\otimes r_{s-1}\otimes\omega_{s-1}\otimes r_s\otimes x+r'r_1\otimes \omega_1\otimes\cdots\otimes r_{s-1}\otimes\omega_{s-1}\otimes r_s\otimes x\\
		=&r(r_1\otimes \omega_1\otimes\cdots\otimes r_{s-1}\otimes\omega_{s-1}\otimes r_s\otimes x)+r'(r_1\otimes \omega_1\otimes\cdots\otimes r_{s-1}\otimes\omega_{s-1}\otimes r_s\otimes x).
	\end{align*}
Also
	\begin{align*}
		&(rr')(r_1\otimes \omega_1\otimes\cdots\otimes r_{s-1}\otimes\omega_{s-1}\otimes r_s\otimes x)\\
		=&(rr')r_1\otimes \omega_1\otimes\cdots\otimes r_{s-1}\otimes\omega_{s-1}\otimes r_s\otimes x\\
		=&r(r'r_1\otimes \omega_1\otimes\cdots\otimes r_{s-1}\otimes\omega_{s-1}\otimes r_s\otimes x)\\
		=&r(r'(r_1\otimes \omega_1\otimes\cdots\otimes r_{s-1}\otimes\omega_{s-1}\otimes r_s\otimes x)).
	\end{align*}
Moreover
	\begin{align*}
		{\bf 1}_R(r_1\otimes \omega_1\otimes\cdots\otimes r_{s-1}\otimes\omega_{s-1}\otimes r_s\otimes x)=r_1\otimes \omega_1\otimes\cdots\otimes r_{s-1}\otimes\omega_{s-1}\otimes r_s\otimes x.
	\end{align*}
	Then $M_{R}(X)$ is a left $R$-module.
\end{proof}

Define a set $\mathfrak{m}_{\Omega}':=(\mathfrak{m}_{\omega}')_{\omega\in \Omega}$ of linear operators $\mathfrak{m}_{\omega}':M_R(X)\longrightarrow M_R(X)$ by assigning
\begin{equation*}
	r_1\otimes\omega_1\otimes\cdots\otimes r_{n-1}\otimes\omega_{n-1}\otimes r_n\otimes x \longmapsto {\bf 1}_R\otimes\omega\otimes r_1\otimes\omega_1\otimes\cdots\otimes r_{n-1}\otimes\omega_{n-1}\otimes r_n\otimes x, .
\end{equation*}
for $r_1\otimes\omega_1\otimes\cdots\otimes r_{n-1}\otimes\omega_{n-1}\otimes r_n\otimes x\in M_R(X)$
and extend them by additivity.

\begin{prop}{\label{B_1}}
	Let $(R,\Omega)$ an $\Omega$-operated algebra and $X$ a set. Then, with the above notions,
	\begin{enumerate}
		\item $(M_R(X),\mathfrak{m}_{\Omega}')$ is a left $\Omega$-operated module;
		\item Let $j_X:X\longrightarrow M_R(X)$, $x\longmapsto {\bf 1}_R\otimes x$, be a set map for any $x\in X$. Then $(M_R(X),\mathfrak{m}_{\Omega}')$ is the free $\Omega$-operated module on $X$.
		% In other words, for any $\Omega$-operated module $(M,\mathfrak{m}_{\Omega})$ equipped with any set map $\varphi:X\longrightarrow M$, there exists a unique left $\Omega$-operated module homomorphism $\overline{\varphi}:M_R(X)\longrightarrow M$ such that $\overline{\varphi} \circ j_X=\varphi$.
	\end{enumerate}
	%        this is represented by the commutative diagram:
	% \begin{displaymath}
		% 		\xymatrix{
			% 			X \ar[dr]_-{\varphi} \ar[r]^-{j_X} & (M_R(X),\mathfrak{m}_{\Omega}') \ar@{-->}[d]^-{\overline{\varphi}} \\
			% 		& (M,\mathfrak{m}_{\Omega}).
			% 	}
		% \end{displaymath}
\end{prop}

\begin{proof}
	(a) We need to prove that the {\bf k}-linearity of $\mathfrak{m}_{\omega}'$ when acting on the tensor product,
	\begin{equation*}
		\begin{split}
			&{\bf 1}_R\otimes k\omega\otimes r_1\otimes\omega_1\otimes\cdots\otimes r_{n-1}\otimes\omega_{n-1}\otimes r_n\otimes x\\
			=&k{\bf 1}_R\otimes\omega\otimes r_1\otimes\omega_1\otimes\cdots\otimes r_{n-1}\otimes\omega_{n-1}\otimes r_n\otimes x,
		\end{split}
	\end{equation*}
	for any $k\in {\bf k}$, $r_1\otimes\omega_1\otimes\cdots\otimes r_{n-1}\otimes\omega_{n-1}\otimes r_n\otimes x\in M_R(X)$.
	
	(b) For any left $\Omega$-operated module $(M,\mathfrak{m}_{\Omega})$ with any set map $\varphi:X\longrightarrow M$. We show that there is a unique left $\Omega$-operated module homomorphism $\overline{\varphi}:M_R(X)\longrightarrow M$ such that $\overline{\varphi}\circ j_X=\varphi$.
	
	We define $\overline{\varphi}:M_R(X)\longrightarrow M$ recursively on the pure tensors $r_1\otimes\omega_1\otimes\cdots\otimes r_{n-1}\otimes\omega_{n-1}\otimes r_n\otimes x\in M_R(X)$. If $n = 1$, then define
	\begin{equation*}
		\overline{\varphi}(r\otimes x)=\overline{\varphi}(r({\bf 1}_R\otimes x))=r\overline{\varphi}({\bf 1}_R\otimes x)=r\overline{\varphi}\circ j_X(x)=r\varphi(x) .
	\end{equation*}
	For $n\geqslant 1$, define $\overline{\varphi}$ recursively as
	\begin{align*}
		&\overline{\varphi}(r_1\otimes\omega_1\otimes r_2\otimes\omega_2\otimes\cdots\otimes x)\\
		=&\overline{\varphi}(r_1({\bf 1}_R\otimes \omega_1\otimes r_2\otimes\omega_2\otimes\cdots\otimes x))\\
		=&r_1\overline{\varphi}({\bf 1}_R\otimes \omega_1\otimes r_2\otimes\omega_2\otimes\cdots\otimes x)\\
		=&r_1\overline{\varphi}(\mathfrak{m}_{\omega_1}'(r_2\otimes\omega_2\otimes\cdots\otimes x))\\
		=&r_1\mathfrak{m}_{\omega_1}\overline{\varphi}(r_2\otimes\omega_2\otimes\cdots\otimes x).
	\end{align*}
	By the construction of $\overline{\varphi}$, it is a left $\Omega$-operated $R$-module homomorphism with $\overline{\varphi}\circ j_X=\varphi$. Moreover, $\overline{\varphi}$ is completely determined by the conditions $\overline{\varphi}\circ j_X=\varphi$ and $\mathfrak{m}_{\omega}\circ \overline{\varphi}=\overline{\varphi} \circ \mathfrak{m}_{\omega}'$. Hence it is unique.
\end{proof}

\begin{defn}{\label{2.10}}
	Let $X$ be a set. A {\bf free left $(R,P_{\Omega})$-module} on $X$ is a left $(R,P_{\Omega})$-module $(F(X),\mathfrak{p}_{\Omega})$ equipped with a set map $j_X:X\longrightarrow F(X)$ satisfying the universal property: for any left $(R,P_{\Omega})$-module $(M,\mathfrak{m}_{\Omega})$ and any set map $\varphi:X\longrightarrow M$, there exists a unique left $(R,P_{\Omega})$-module homomorphism $\widetilde{\varphi}:F(X)\longrightarrow M$ such that $\widetilde{\varphi}\circ j_X=\varphi$.
\end{defn}

By the properties of free modules, we obtain
\begin{prop}{\label{BB}}
	\begin{enumerate}
		\item Any left multiple Rota-Baxter module is isomorphic to the quotient module of a free left multiple Rota-Baxter module.
		
		\item Any finitely generated left multiple Rota-Baxter module is isomorphic to the quotient module of the finitely generated free left multiple Rota-Baxter module.
	\end{enumerate}
\end{prop}

Next we construct the free left $(R,P_{\Omega})$-module on a set $X$. Let $I_X$ be the left $\Omega$-operated submodule of $M_R(X)$ generated by the set
\begin{equation*}
	\left\{P_{\alpha}(r)\mathfrak{m}_{\beta}'(a)-\mathfrak{m}_{\alpha}'\left(r\mathfrak{m}_{\beta}'(a)\right)-\mathfrak{m}_{\beta}'(P_{\alpha}(r)a)-\lambda_{\beta}\mathfrak{m}_{\alpha}'(ra)-\lambda_{\alpha}\mathfrak{m}_{\beta}'(ra)\ |\ r\in R,a\in M_{R}(X),\alpha,\beta\in\Omega \right\}.
\end{equation*}
Take $M_R(X)/I_X$ as the quotient $R$-module of $M_R(X)$ by $I_X$ and for each $\omega\in\Omega$, define
\begin{equation*}
	\overline{\mathfrak{m}}_{\omega}':M_R(X)/I_X\longrightarrow M_R(X)/I_X,\quad a+I_X\longmapsto \mathfrak{m}_{\omega}'(a)+I_X,
\end{equation*}
to be the operator on $M_R(X)/I_X$ induced by $\mathfrak{m}_{\omega}$.
For $r\in R,a\in M_{R}(X),\alpha,\beta\in\Omega$, we have
\begin{equation*}
	\begin{split}
		P_{\alpha}(r)\overline{\mathfrak{m}}_{\beta}'(a+I_X)=&P_{\alpha}(r)(\mathfrak{m}_{\beta}'(a)+I_X)\\
		=&P_{\alpha}(r)\mathfrak{m}_{\beta}'(a)+I_X\\
		=&(\mathfrak{m}_{\alpha}'\left(r\mathfrak{m}_{\beta}'(a)\right)+\mathfrak{m}_{\beta}'(P_{\alpha}(r)a)+\lambda_{\beta}\mathfrak{m}_{\alpha}'(ra)+\lambda_{\alpha}\mathfrak{m}_{\beta}'(ra))+I_X\\
		=&(\mathfrak{m}_{\alpha}'\left(r\mathfrak{m}_{\beta}'(a)\right)+I_X)+(\mathfrak{m}_{\beta}'(P_{\alpha}(r)a)+I_X)
		+(\lambda_{\beta}\mathfrak{m}_{\alpha}'(ra)+I_X)+(\lambda_{\alpha}\mathfrak{m}_{\beta}'(ra)+I_X)\\
		=&\overline{\mathfrak{m}}_{\alpha}'\left(r\mathfrak{m}_{\beta}'(a)+I_X\right)+\overline{\mathfrak{m}}_{\beta}'(P_{\alpha}(r)a+I_X)
		+\lambda_{\beta}\overline{\mathfrak{m}}_{\alpha}'(ra+I_X)+\lambda_{\alpha}\overline{\mathfrak{m}}_{\beta}'(ra+I_X)\\
		=&\overline{\mathfrak{m}}_{\alpha}'\left(r\overline{\mathfrak{m}}_{\beta}'(a+I_X)\right)+\overline{\mathfrak{m}}_{\beta}'(P_{\alpha}(r)(a+I_X))
		+\lambda_{\beta}\overline{\mathfrak{m}}_{\alpha}'(r(a+I_X))+\lambda_{\alpha}\overline{\mathfrak{m}}_{\beta}'(r(a+I_X)).
	\end{split}
\end{equation*}
Thus $(M_R(X)/I_X,\overline{\mathfrak{m}}_{\Omega}')$ is a left $(R,P_{\Omega})$-module.

\begin{theorem}{\label{C}}
	Let $X$ be a set. Then $(M_R(X)/I_X,\overline{\mathfrak{m}}_{\Omega}')$ is the free left $(R,P_{\Omega})$-module with the natural map $j:=\pi\circ j_X:X\longrightarrow M_R(X)\longrightarrow M_R(X)/I_X$.
\end{theorem}

\begin{proof}
	For a left $(R,P_{\Omega})$-module $(M,\mathfrak{m}_{\Omega})$ and a set map $\varphi:X\longrightarrow M$. We show that there exists a unique left $(R,P_{\Omega})$-module homomorphism $\widetilde{\varphi}:(M_R(X)/I_X,\overline{\mathfrak{m}}_{\Omega}')\longrightarrow(M,\mathfrak{m}_{\Omega})$ such that $\widetilde{\varphi}\circ j=\varphi$. By Proposition~\ref{B_1}, there exists a unique left $\Omega$-operated module homomorphism $\overline{\varphi}:(M_R(X),\mathfrak{m}_{\Omega}')\longrightarrow(M,\mathfrak{m}_{\Omega})$ such that $\overline{\varphi}\circ j_X=\varphi$. That is represented by the commutative diagram:
	\begin{displaymath}
		\xymatrix{
			X \ar[dr]_-{\varphi} \ar[r]^-{j_X} & (M_R(X),\mathfrak{m}_{\Omega}') \ar[d]^-{\overline{\varphi}} \ar[r]^-{\pi} &(M_R(X)/I_X,\overline{\mathfrak{m}}_{\Omega}') \ar@{-->}[dl]^-{\widetilde{\varphi}} \\
			& (M,\mathfrak{m}_{\Omega}).
		}
	\end{displaymath}
	Next we prove that $\widetilde{\varphi}$ vanishes on $I_X$,
	\begin{equation*}
		\begin{split}
			&\overline{\varphi}\left(P_{\alpha}(r)\mathfrak{m}_{\beta}'(a)-\mathfrak{m}_{\alpha}'\left(r\mathfrak{m}_{\beta}'(a)\right)-\mathfrak{m}_{\beta}'\left(P_{\alpha}(r)a\right)-\lambda_{\beta}\mathfrak{m}_{\alpha}'\left(ra\right)-\lambda_{\alpha}\mathfrak{m}_{\beta}'\left(ra\right)\right)\\
			=&P_{\alpha}\left(r\right)\mathfrak{m}_{\beta}(\overline{\varphi}(a))-\mathfrak{m}_{\alpha}\left(r\mathfrak{m}_{\beta}(\overline{\varphi}(a))\right)-\mathfrak{m}_{\beta}\left(P_{\alpha}r\overline{\varphi}(a)\right)\\
			&-\lambda_{\beta}\mathfrak{m}_{\alpha}\left(r\overline{\varphi}(a)\right)-\lambda_{\alpha}\mathfrak{m}_{\beta}\left(r\overline{\varphi}(a)\right)\\
			=&0.
		\end{split}
	\end{equation*}
	Since $\overline{\varphi}$ vanishes on the generators of $I_X$ and $\overline{\varphi}$ is a left $\Omega$-operated module homomorphism, $\overline{\varphi}(I_X)=0$. Then $\overline{\varphi}$ induces a unique left $\Omega$-operated module homomorphism
	\begin{equation*}
		\widetilde{\varphi}:(M_R(X)/I_X,\overline{\mathfrak{m}}_{\Omega}')\longrightarrow(M,\mathfrak{m}_{\Omega}),
	\end{equation*}
	such that $\widetilde{\varphi}\circ \pi=\overline{\varphi}$.
	
	Since $(M_R(X)/I_X,\overline{\mathfrak{m}}_{\Omega}')$ and $(M,\mathfrak{m}_{\Omega})$ are left $(R,P_{\Omega})$-modules and for each $\omega\in\Omega$, $\widetilde{\varphi}\circ \overline{\mathfrak{m}}_{\omega}'=\mathfrak{m}_{\omega}\circ \widetilde{\varphi}$, $\widetilde{\varphi}$ is the unique left $(R,P_{\Omega})$-module homomorphism such that
	\begin{equation*}
		\widetilde{\varphi}\circ j=\widetilde{\varphi}\circ \pi\circ j_X=\overline{\varphi}\circ j_X=\varphi.
	\end{equation*}	
This completes the proof.
\end{proof}

\subsection{Free modules as free multiple Rota-Baxter modules}
%A multiple Rota-Baxter algebra $(R, P_{\Omega})$ is typically not free when viewed as a module over itself,
In this subsection, we give strict condition under which any free $R$-module is a free $(R, P_{\Omega})$-module.

Let $X$ be a set. For any left $(R,P_{\Omega})$-module $(M,\mathfrak{m}_{\Omega})$, we define the set of {\bf module constant} of $M$ as
\begin{equation*}
	MC(M):=\left\{m\in M|\mathfrak{m}_{\omega}(rm)=P_{\omega}(r)m ,\quad \text{for all } r\in R,\omega\in\Omega\right\},
\end{equation*}
where $m\in MC(M)$ acts as a constant under the operator action.
\begin{defn}
	Let $X$ be a set. A {\bf restricted free left $(R,P_{\Omega})$-module} on $X$ is a left $(R,P_{\Omega})$-module $(F(X),\mathfrak{p}_{\Omega})$ equipped a set map $j_X:X\longrightarrow F(X)$ satisfying, for any left $(R,P_{\Omega})$-module $(M,\mathfrak{m}_{\Omega})$ and any set map $\varphi:X\longrightarrow M$, if $im(\varphi)\subseteq MC(M)$, there exists a unique left $(R,P_{\Omega})$-module homomorphism $\overline{\varphi}:(F(X),\mathfrak{p}_{\Omega})\longrightarrow (M,\mathfrak{m}_{\Omega})$ such that $\overline{\varphi}\circ j_X=\varphi$.
\end{defn}

Let $\widetilde{F}(X)$ be the free left $R$-module on $X$
\begin{equation*}
	\widetilde{F}(X):=\left\{\sum_{x\in X}^{}r_xx|r_x\in R\right\}.
\end{equation*}
Define linear operators $\widetilde{\mathfrak{p}}_{\Omega}:=(\widetilde{\mathfrak{p}}_{\omega})_{\omega\in\Omega}$ on $\widetilde{F}(X)$ as follows,
\begin{equation*}
	\widetilde{\mathfrak{p}}_{\omega}:\widetilde{F}(X)\longrightarrow\widetilde{F}(X),\quad \sum_{x\in X}^{}r_xx\longmapsto\sum_{x\in X}^{}P_{\omega}(r_x)x.
\end{equation*}

\begin{theorem}\label{2.16}
	Let $X$ be a set. Then
	\begin{enumerate}
		\item $(\widetilde{F}(X),\widetilde{\mathfrak{p}}_{\Omega})$ is a left $(R,P_{\Omega})$-module.
		\item $(\widetilde{F}(X),\widetilde{\mathfrak{p}}_{\Omega})$ equipped with the natural embedding map $j:X\longrightarrow(\widetilde{F}(X),\widetilde{\mathfrak{p}}_{\Omega})$ is the restricted free left $(R,P_{\Omega})$-module on $X$ .
		% In other words, for any left $(R,P_{\Omega})$-module $(M,\mathfrak{m}_{\Omega})$ and any set map $\varphi:X\longrightarrow (M,\mathfrak{m}_{\Omega}) $ satisfying $im(\varphi) \subseteq MC(M)$, there exists a unique left $(R,P_{\Omega})$-module homomorphism $\overline{\varphi}:(\widetilde{F}(X),\widetilde{P}_{\Omega})\longrightarrow (M,\mathfrak{m}_{\Omega})$ such that $\varphi=\overline{\varphi}\circ j$.
	\end{enumerate}
\end{theorem}
\begin{proof}
	(a) We need to prove $\widetilde{\mathfrak{p}_{\omega}}$ satisfies Eq.~(\ref{a}). For all $r\in R$, $r_xx\in \widetilde{F}(X)$ and $\alpha,\beta\in\Omega$,
	\begin{align*}
		P_{\alpha}(r)\widetilde{\mathfrak{p}}_{\beta}(r_xx)&=P_{\alpha}(r)(P_{\beta}(r_x)x)\\
		&=P_{\alpha}(r)P_{\beta}(r_x)x\\
		&=P_{\alpha}(rP_{\beta}(r_x))x+P_{\beta}(P_{\alpha}(r)r_x)x+\lambda_{\alpha}P_{\alpha}(rr_x)x+\lambda_{\alpha}P_{\beta}(rr_x)x\\
		&=\widetilde{\mathfrak{p}}_{\alpha}(rP_{\beta}(r_x)x)+\widetilde{\mathfrak{p}}_{\beta}(P_{\alpha}(r)r_xx)+\lambda_{\beta}\widetilde{\mathfrak{p}}_{\alpha}(rr_xx)+\lambda_{\alpha}\widetilde{\mathfrak{p}}_{\beta}(rr_xx)\\
		&=\widetilde{\mathfrak{p}}_{\alpha}(r\widetilde{\mathfrak{p}}_{\beta}(r_xx))+\widetilde{\mathfrak{p}}_{\beta}(P_{\alpha}(r)(r_xx))+\lambda_{\beta}\widetilde{\mathfrak{p}}_{\alpha}(r(r_xx))+\lambda_{\alpha}\widetilde{\mathfrak{p}}_{\beta}(r(r_xx)),
	\end{align*}
	as desired.
	
	(b) For any left $(R,P_{\Omega})$-module $(M,\mathfrak{m}_{\Omega})$, any set map $\varphi:X\longrightarrow M$ and $im(\varphi)\subseteq MC(M)$, we show that there exists a unique left $(R,P_{\Omega})$-module homomorphism $\overline{\varphi}:(\widetilde{F}(X),\widetilde{\mathfrak{p}}_{\Omega})\longrightarrow(M,\mathfrak{m}_{\Omega})$ such that $\overline{\varphi}\circ j_X=\varphi$.
	By the universal property of free $R$-module $\widetilde{F}(X)$, there is a left $R$-module homomorphism
	\begin{equation*}
		\overline{\varphi}:(\widetilde{F}(X),\widetilde{\mathfrak{p}}_{\Omega})\longrightarrow (M,\mathfrak{m}_{\Omega}),\quad \sum_{x\in X}^{}r_xx\longmapsto\sum_{x\in X}^{}r_x\varphi(x).
	\end{equation*}
	Next we prove that $\overline{\varphi}$ is the unique required left $(R,P_{\Omega})$-module homomorphism.
	
	%    We consider the free left $R$-module $\widetilde{F}(X)$ over $X$. We obtain a induced left $R$-module homomorphism
	% \begin{equation}{\label{c}}
		% 	\overline{\varphi}:(\widetilde{F}(X),\widetilde{\mathfrak{p}}_{\Omega})\longrightarrow (M,\mathfrak{m}_{\Omega}),\quad \sum_{x\in X}^{}r_xx\longmapsto\sum_{x\in X}^{}r_x\varphi(x).
		% \end{equation}
	For $r_xx\in \widetilde{F}(X)$, $\omega\in\Omega$, we have
	\begin{align*}
		(\overline{\varphi}\circ \widetilde{\mathfrak{p}}_{\omega})(r_xx)&=\overline{\varphi}(\widetilde{\mathfrak{p}}_{\omega}(r_xx))\\
		&=\overline{\varphi}(P_{\omega}(r_x)x)\quad(\text{ by the definition of $\widetilde{\mathfrak{p}}_{\omega}$})\\
		&=P_{\omega}(r_x)\varphi(x)\quad (\text{by the definition of $\overline{\varphi}$})\\
		&=\mathfrak{m}_{\omega}(r_x\varphi(x))\quad(\text{by $im(\varphi)\subseteq MC(M)$})\\
		&=\mathfrak{m}_{\omega}(\overline{\varphi}(r_xx))\quad (\text{by the definition of $\overline{\varphi}$})\\
		&=(\mathfrak{m}_{\omega}\circ\overline{\varphi})(r_xx).
	\end{align*}
	Thus $\overline{\varphi}$ is the required left $(R,P_{\Omega})$-module homomorphism.
	
	Since
	\begin{equation*}
		\overline{\varphi}(r_xx)=r_x\varphi(x)=r_x((\overline{\varphi}\circ j_X)(x))=r_x\varphi(x),
	\end{equation*}
	$\overline{\varphi}$ is completely determined by $\varphi$. Thus the uniqueness of $\overline{\varphi}$ is proven.
\end{proof}

Take $X=\{x\}$ be a singleton set, define left $(R,P_{\Omega})$-module homomorphisms
\begin{equation*}
	f:(R,P_{\Omega})\longrightarrow (\widetilde{F}(X),\widetilde{\mathfrak{p}}_{\Omega}),\quad r\longmapsto rx,
\end{equation*}
and
\begin{equation*}
	g:(\widetilde{F}(X),\widetilde{\mathfrak{p}}_{\Omega})\longrightarrow (R,P_{\Omega}),\quad rx\longmapsto r,
\end{equation*}
for all $r\in R$. Then it is easy to check that $f\circ g=id_{(\widetilde{F}(X),\widetilde{\mathfrak{p}}_{\Omega})}$ and $g\circ f=id_{(R,P_{\Omega})}$. Hence there exists left $(R,P_{\Omega})$-module isomorphism $(R,P_{\Omega})\cong (\widetilde{F}(X),\widetilde{\mathfrak{p}}_{\Omega})$. By Theorem \ref{2.16}, we have
\begin{coro}
	Let $X$ be a singleton, $(R,P_{\Omega})$ is the restricted free left $(R,P_{\Omega})$-module.
\end{coro}

\section{Projective and injective multiple Rota-Baxter module}\label{3}
In this section, we focus on the $\mathrm{Hom}$ functor, the projective and the injective properties in the category of multiple Rota-Baxter modules.

\subsection{The Hom functor}Denote the category of left $(R,P_{\Omega})$-modules by $_{(R,P_{\Omega})}{\bf Mod}$. For any two objects $(M,\mathfrak{m}_{\Omega})$ and $(N,\mathfrak{n}_{\Omega})$ in $_{(R,P_{\Omega})}{\bf Mod}$, we define $\mathrm{Hom}_{(R,P_{\Omega})}(M,N)$ as the collection of all the left $(R,P_{\Omega})$-modules homomorphisms from $(M,\mathfrak{m}_{\Omega})$ to $(N,\mathfrak{n}_{\Omega})$. Thus $\mathrm{Hom}_{(R,P_{\Omega})}(M,N)$ is a subset of $\mathrm{Hom}_R(M,N)$.

\begin{prop}{\label{D}}
	Let $(M,\mathfrak{m}_{\Omega}),(N,\mathfrak{n}_{\Omega})$ be left $(R,P_{\Omega})$-modules. Then $\mathrm{Hom}_{(R,P_{\Omega})}(M,N)$ is an abelian subgroup of $\mathrm{Hom}_R(M,N)$.
\end{prop}

\begin{proof}
	For any $f_1,f_2\in \mathrm{Hom}_{(R,P_{\Omega})}(M,N)$, since $\mathrm{Hom}_{(R,P_{\Omega})}(M,N)$ is a subset of $\mathrm{Hom}_R(M,N)$, $f_1-f_2\in \mathrm{Hom}_R(M,N)$.
	Next we show $f_1-f_2\in \mathrm{Hom}_{(R,P_{\Omega})}(M,N)$.
	
	Since $f_1\circ \mathfrak{m}_{\omega}=\mathfrak{n}_{\omega}\circ f_1$ and $f_2\circ \mathfrak{m}_{\omega}=\mathfrak{n}_{\omega}\circ f_2$, we have
	\begin{equation*}
		(f_1-f_2)\circ \mathfrak{m}_{\omega}=f_1\circ \mathfrak{m}_{\omega}+(-f_2)\circ\mathfrak{m}_{\omega}=\mathfrak{n}_{\omega}\circ f_1+\mathfrak{n}_{\omega}\circ (-f_2)=\mathfrak{n}_{\omega}\circ (f_1-f_2).
	\end{equation*}
	Thus $f_1-f_2\in \mathrm{Hom}_{(R,P_{\Omega})}(M,N)$. Hence $\mathrm{Hom}_{(R,P_{\Omega})}(M,N)$ is an abelian subgroup of $\mathrm{Hom}_R(M,N)$.
\end{proof}

\begin{prop}
	Let $(R',P_{\Omega}')$ and $(R'',P_{\Omega}'')$ be multiple Rota-Baxter algebras of pair weight $(\lambda_{\Omega},\lambda_{\Omega})$.
	\begin{enumerate}
		\item If $(M_{(R,P_{\Omega})},\mathfrak{m}_{\Omega}^R)$ is a right $(R,P_{\Omega})$-module and $(_{(R'',P_{\Omega}'')}N_{(R,P_{\Omega})},\mathfrak{n}_{\Omega}^{R''},\mathfrak{n}_{\Omega}^R)$ is an $(R'',P_{\Omega}'')$-$(R,P_{\Omega})$-bimodule, then $(\mathrm{Hom}_{(R,P_{\Omega})}(M,N),\mathfrak{q}_{\Omega})$ is a left $(R'',P_{\Omega}'')$-module with $\mathfrak{q}_{\Omega}:=(\mathfrak{q}_{\omega})_{\omega\in \Omega}$ defined by
		\begin{equation*}
			\mathfrak{q}_{\omega}(\varphi)(m):=\mathfrak{n}_{\omega}^{R''}(\varphi(m)),\quad \varphi\in \mathrm{Hom}_{(R,P_{\Omega})}(M,N),\quad m\in M.
		\end{equation*}
		\item If $(_{(R,P_{\Omega})}M,\mathfrak{m}_{\Omega}^R)$ is a left $(R,P_{\Omega})$-module and $(_{(R,P_{\Omega})}N_{(R'',P_{\Omega}'')},\mathfrak{n}_{\Omega}^R,\mathfrak{n}_{\Omega}^{R''})$ is an $(R,P_{\Omega})$-$(R'',P_{\Omega}'')$-bimodule, then $(\mathrm{Hom}_{(R,P_{\Omega})}(M,N),\mathfrak{q}_{\Omega})$ is a right $(R'',P_{\Omega}'')$-module with $\mathfrak{q}_{\Omega}:=(\mathfrak{q}_{\omega})_{\omega\in \Omega}$ defined by
		\begin{equation*}
			\mathfrak{q}_{\omega}(\varphi)(m):=\mathfrak{n}_{\omega}^{R''}(\varphi(m)),\quad \varphi\in \mathrm{Hom}_{(R,P_{\Omega})}(M,N),\quad m\in M.
		\end{equation*}
		\item If $(_{(R,P_{\Omega})}M_{(R',P_{\Omega}')},\mathfrak{m}_{\Omega}^R,\mathfrak{m}_{\Omega}^{R'})$ is an $(R,P_{\Omega})$-$(R',P_{\Omega}')$-bimodule and $(_{(R,P_{\Omega})}N,\mathfrak{n}_{\Omega}^R)$ is a left $(R,P_{\Omega})$-module, then $(\mathrm{Hom}_{(R,P_{\Omega})}(M,N),\mathfrak{q}_{\Omega})$ is a left $(R',P_{\Omega}')$-module with $\mathfrak{q}_{\Omega}:=(\mathfrak{q}_{\omega})_{\omega\in \Omega}$ defined by
		\begin{equation*}
			\mathfrak{q}_{\omega}(\varphi)(m):=\varphi(\mathfrak{m}_{\omega}^{R'}(m)),\quad \varphi\in \mathrm{Hom}_{(R,P_{\Omega})}(M,N),\quad m\in M.
		\end{equation*}
		\item If $(_{(R',P_{\Omega}')}M_{(R,P_{\Omega})},\mathfrak{m}_{\Omega}^{R'},\mathfrak{m}_{\Omega}^R)$ is an $(R',P_{\Omega}')$-$(R,P_{\Omega})$-bimodule and $(N_{(R,P_{\Omega})},\mathfrak{n}_{\Omega}^R)$ is a right $(R,P_{\Omega})$-module, then $(\mathrm{Hom}_{(R,P_{\Omega})}(M,N),\mathfrak{q}_{\Omega})$ is a right $(R',P_{\Omega}')$-module with $\mathfrak{q}_{\Omega}:=(\mathfrak{q}_{\omega})_{\omega\in \Omega}$ defined by
		\begin{equation*}
			\mathfrak{q}_{\omega}(\varphi)(m):=\varphi(\mathfrak{m}_{\omega}^{R'}(m)),\quad \varphi\in \mathrm{Hom}_{(R,P_{\Omega})}(M,N),\quad m\in M.
		\end{equation*}
	\end{enumerate}
\end{prop}

\begin{proof}
	(a). For any $r''\in R''$ and $\varphi\in \mathrm{Hom}_{(R,P_{\Omega})}(M,N)$, the left $R''$-module structure on $\Hom_{(R,P_{\Omega})}(M,N)$ is defined by
	\begin{equation*}
		(r''\varphi)(m):=r''\varphi(m),\quad m\in M.
	\end{equation*}
	% For any $\varphi\in\mathrm{Hom}_{(R,P_{\Omega})}(M,N)$, we have $\varphi\circ \mathfrak{m}_{\omega}^R=\mathfrak{n}_{\omega}^R\circ \varphi$.
	Moreover, for each $\omega\in\Omega$,
	\begin{equation*}
		\begin{split}
			(r''\varphi)\circ \mathfrak{m}_{\omega}^R(m)=&r''\varphi(\mathfrak{m}_{\omega}^R(m))\\
			=&r''\mathfrak{n}_{\omega}^R(\varphi(m))\quad (\text{by $\varphi$ being a right $(R,P_{\Omega})$-module homomorphism})\\
			=&\mathfrak{n}_{\omega}^R(r''\varphi(m))\quad (\text{by $(_{(R'',P_{\Omega}'')}N_{(R,P_{\Omega})},\mathfrak{n}_{\Omega}^{R''},\mathfrak{n}_{\Omega}^R)$ being an $(R'',P_{\Omega}'')$-$(R,P_{\Omega})$-bimodule})\\
			=&\mathfrak{n}_{\omega}^R\circ (r''\varphi)(m).
		\end{split}
	\end{equation*}
	%    Since $N$ carries a structure of $(R'',P_{\Omega}'')$-$(R,P_{\Omega})$-bimodule,
	% \begin{equation*}
		% 	(r''\varphi)\circ \mathfrak{m}_{\omega}^R(m)=r''\varphi(\mathfrak{m}_{\omega}^R(m))=r''\mathfrak{n}_{\omega}^R(\varphi(m))=\mathfrak{n}_{\omega}^R(r''\varphi(m)),\quad \forall m\in M.
		% \end{equation*}
	Hence $\mathrm{Hom}_{(R,P_{\Omega})}(M,N)$ is a left $R''$-submodule of $\mathrm{Hom}_R(M,N)$.
	
	Now we prove for each $\omega\in\Omega$, $\mathfrak{q}_{\omega}(\varphi)\in \mathrm{Hom}_{(R,P_{\Omega})}(M,N)$. For $m\in M$,
	\begin{equation*}
		\begin{split}
			\mathfrak{q}_{\omega}(\varphi)\circ \mathfrak{m}_{\omega}^{R}(m)=&\mathfrak{q}_{\omega}(\varphi)(\mathfrak{m}_{\omega}^{R}(m))\\
			=&\mathfrak{n}_{\omega}^{R''}(\varphi(\mathfrak{m}_{\omega}^{R}(m)))\quad (\text{by the definition of $\mathfrak{q}_{\omega}(\varphi)$})\\
			=&\mathfrak{n}_{\omega}^{R''}(\mathfrak{n}_{\omega}^R(\varphi(m)))\quad (\text{by $\varphi$ being a right $(R,P_{\Omega})$-module homomorphism})\\
			=&\mathfrak{n}_{\omega}^R(\mathfrak{n}_{\omega}^{R''}(\varphi(m)))\quad (\text{by $(_{(R'',P_{\Omega}'')}N_{(R,P_{\Omega})},\mathfrak{n}_{\Omega}^{R''},\mathfrak{n}_{\Omega}^R)$ being an $(R'',P_{\Omega}'')$-$(R,P_{\Omega})$-bimodule})\\
			=&\mathfrak{n}_{\omega}^R(\mathfrak{q}_{\omega}(\varphi)(m))\\
			=&\mathfrak{n}_{\omega}^R\circ \mathfrak{q}_{\omega}(\varphi)(m).
		\end{split}
	\end{equation*}
	Hence $\mathfrak{q}_{\omega}(\varphi)$ is also a right $R$-module homomorphism.
	% Because $\varphi$ and $\mathfrak{n}_{\omega}^{R''}$ are right $R$-module homomorphisms, the composition $\mathfrak{q}_{\omega}(\varphi)$ is also a right $R$-module homomorphism .
	
	Finally we need to prove $\mathfrak{q}_{\omega}(\varphi)$ satisfies Eq.~(\ref{a}).
	% \begin{equation*}
		% 	P''_{\alpha}(r'')\mathfrak{q}_{\beta}(\varphi)=\mathfrak{q}_{\alpha}(r''\mathfrak{q}_{\beta}(\varphi))+\mathfrak{q}_{\beta}(P''_{\alpha}(r'')\varphi)+\lambda_{\beta}\mathfrak{q}_{\alpha}(r''\varphi)+\lambda_{\alpha}\mathfrak{q}_{\beta}(r''\varphi)\quad \forall\ r''\in R''.
		% \end{equation*}
	For $\alpha,\beta\in\Omega, m\in M, r''\in R''$ we have
	\begin{align*}
		(P''_{\alpha}(r'')\mathfrak{q}_{\beta}(\varphi))(m)=&P''_{\alpha}(r'')(\mathfrak{q}_{\beta}(\varphi)(m))\\
		=&P''_{\alpha}(r'')\mathfrak{n}_{\beta}^{R''}(\varphi(m))\\
		=&\mathfrak{n}_{\alpha}^{R''}(r''\mathfrak{n}_{\beta}^{R''}(\varphi(m)))+\mathfrak{n}_{\beta}^{R''}(P''_{\alpha}(r'')\varphi(m))+\lambda_{\beta}\mathfrak{n}_{\alpha}^{R''}(r''\varphi(m))+\lambda_{\alpha}\mathfrak{n}_{\beta}^{R''}(r''\varphi(m))\\
		&\ \text{(by $(N,\mathfrak{n}_{\Omega})$ being a left $(R'',P_{\Omega}'')$-module)}\\
		=&\mathfrak{n}_{\alpha}^{R''}(r''(\mathfrak{q}_{\beta}\varphi)(m))+\mathfrak{n}_{\beta}^{R''}((P''_{\alpha}(r'')\varphi)(m))\\
		&+\lambda_{\beta}\mathfrak{n}_{\alpha}^{R''}((r''\varphi)(m))+\lambda_{\alpha}\mathfrak{n}_{\beta}^{R''}((r''\varphi)(m))\\
		=&\mathfrak{q}_{\alpha}(r''\mathfrak{q}_{\beta}(\varphi))(m)+\mathfrak{q}_{\beta}(P''_{\alpha}(r'')\varphi)(m)+\lambda_{\beta}\mathfrak{q}_{\alpha}(r''\varphi)(m)+\lambda_{\alpha}\mathfrak{q}_{\beta}(r''\varphi)(m).
	\end{align*}
	Hence
	\begin{equation*}
		P''_{\alpha}(r'')\mathfrak{q}_{\beta}(\varphi)=\mathfrak{q}_{\alpha}(r''\mathfrak{q}_{\beta}(\varphi))+\mathfrak{q}_{\beta}(P''_{\alpha}(r'')\varphi)+\lambda_{\beta}\mathfrak{q}_{\alpha}(r''\varphi)+\lambda_{\alpha}\mathfrak{q}_{\beta}(r''\varphi),
	\end{equation*}
	as required.
	
	(b). The verification proceeds similarly to Item (a) .
	
	(c). For any $r'\in R'$ and $\varphi\in \mathrm{Hom}_{(R,P_{\Omega})}(M,N)$, the left $R'$-module structure on $\mathrm{Hom}_{(R,P_{\Omega})}(M,N)$ is defined by
	\begin{equation*}
		(r'\varphi)(m)=\varphi(mr'),\quad m\in M.
	\end{equation*}
	Moreover, for each $\omega\in\Omega$,
	\begin{equation*}
		\begin{split}
			(r'\varphi)\circ \mathfrak{m}_{\omega}^R(m)=&r'\varphi(\mathfrak{m}_{\omega}^R(m))\\
			=&\varphi(\mathfrak{m}_{\omega}^R(m)r')\\
			=&\varphi(\mathfrak{m}_{\omega}^R(mr'))\quad(\text{by $(_{(R,P_{\Omega})}M_{(R',P_{\Omega}')},\mathfrak{m}_{\Omega}^R,\mathfrak{m}_{\Omega}^{R'})$ being an $(R,P_{\Omega})$-$(R',P_{\Omega}')$-bimodule})\\
			=&\mathfrak{n}_{\omega}^R(\varphi(mr'))\quad (\text{by $\varphi $ being a left $(R,P_{\Omega})$-module homomorphism})\\
			=&\mathfrak{n}_{\omega}^R\circ(r'\varphi)(m).
		\end{split}
	\end{equation*}
	Hence $\mathrm{Hom}_{(R,P_{\Omega})}(M,N)$ is a left $R'$-submodule of $\mathrm{Hom}_R(M,N)$.
	Now we prove for each $\omega\in\Omega$, $\mathfrak{q}_{\omega}(\varphi)\in \mathrm{Hom}_{(R,P_{\Omega})}(M,N)$. For $m\in M$,
	\begin{equation*}
		\begin{split}
			\mathfrak{q}_{\omega}(\varphi)\circ \mathfrak{m}_{\omega}^R(m)
			=&\mathfrak{q}_{\omega}(\varphi)( \mathfrak{m}_{\omega}^R(m))\\
			=&\varphi(\mathfrak{m}_{\omega}^{R'}(\mathfrak{m}_{\omega}^R(m)))\quad (\text{by the definition of $\mathfrak{q}_{\omega}(\varphi)$ })\\
			=&\varphi(\mathfrak{m}_{\omega}^R(\mathfrak{m}_{\omega}^{R'}(m)))\quad (\text{by $(_{(R,P_{\Omega})}M_{(R',P_{\Omega}')},\mathfrak{m}_{\Omega}^R,\mathfrak{m}_{\Omega}^{R'})$ being an $(R,P_{\Omega})$-$(R',P_{\Omega}')$-bimodule})\\
			=&\mathfrak{n}_{\omega}^R(\varphi(\mathfrak{m}_{\omega}^{R'}(m)))\quad (\text{by $\varphi$ being a left $(R,P_{\Omega})$-module homomorphism })\\
			=&\mathfrak{n}_{\omega}^R(\mathfrak{q}_{\omega}(\varphi)(m))\\
			=&\mathfrak{n}_{\omega}^R\circ \mathfrak{q}_{\omega}(\varphi)(m).
		\end{split}
	\end{equation*}
	Hence $\mathfrak{q}_{\omega}(\varphi)$ is also a left $R$-module homomorphism.
	
	Finally we need to prove $\mathfrak{q}_{\omega}(\varphi)$ satisfies Eq.~(\ref{a}). For $\alpha,\beta\in\Omega, m\in M, r'\in R'$,
	%    To prove
	% \begin{equation*}
		% 	P'_{\alpha}(r')\mathfrak{q}_{\beta}(\varphi)=\mathfrak{q}_{\alpha}(r'\mathfrak{q}_{\beta}(\varphi))+\mathfrak{q}_{\beta}(P'_{\alpha}(r')\varphi)+\lambda_{\beta}\mathfrak{q}_{\alpha}(r'\varphi)+\lambda_{\alpha}\mathfrak{q}_{\beta}(r'\varphi),\quad \forall\ r'\in R',
		% \end{equation*}
	we have
	\begin{align*}
		(P'_{\alpha}(r')\mathfrak{q}_{\beta}(\varphi))(m)=&\mathfrak{q}_{\beta}(\varphi)(mP'_{\alpha}(r'))\quad (\text{by\ the\ definition\ of $R'$-action)}\\
		=&\varphi(\mathfrak{m}_{\beta}^{R'}(mP'_{\alpha}(r')))\quad (\text{by\ the\ definition\ of\ $\mathfrak{q}_{\Omega}(\varphi)$})\\
		=&\varphi(\mathfrak{m}_{\beta}^{R'}(\mathfrak{m}_{\alpha}^{R'}(m)r')+\mathfrak{m}_{\beta}^{R'}(m)P'_{\alpha}(r')+\lambda_{\beta}(\mathfrak{m}_{\alpha}^{R'}(m)r')+\lambda_{\alpha}(\mathfrak{m}_{\beta}^{R'}(m)r')\\
		&(\text{by $(M,\mathfrak{m}_{\Omega})$ being a right $(R',P_{\Omega}')$-module})\\
		=&\varphi(\mathfrak{m}_{\beta}^{R'}(\mathfrak{m}_{\alpha}^{R'}(m)r'))+\varphi(\mathfrak{m}_{\beta}^{R'}(m)P'_{\alpha}(r'))\\
		&+\lambda_{\beta}\varphi(\mathfrak{m}_{\alpha}^{R'}(m)r')+\lambda_{\alpha}\varphi(\mathfrak{m}_{\beta}^{R'}(m)r')\\
		=&\mathfrak{q}_{\beta}(\varphi)(\mathfrak{m}_{\alpha}^{R'}(m)r')+(P'_{\alpha}(r')\varphi)(\mathfrak{m}_{\beta}^{R'}(m))\\
		&+\lambda_{\beta}(r'\varphi)(\mathfrak{m}_{\alpha}^{R'}(m))+\lambda_{\alpha}(r'\varphi)(\mathfrak{m}_{\beta}^{R'}(m))\\
		=&(r'\mathfrak{q}_{\beta}(\varphi))(\mathfrak{m}_{\alpha}^{R'}(m))+(P'_{\alpha}(r')\varphi)(\mathfrak{m}_{\beta}^{R'}(m))\\
		&+\lambda_{\beta}(r'\varphi)(\mathfrak{m}_{\alpha}^{R'}(m))+\lambda_{\alpha}(r'\varphi)(\mathfrak{m}_{\beta}^{R'}(m))\\
		=&\mathfrak{q}_{\alpha}(r'\mathfrak{q}_{\beta}(\varphi))(m)+\mathfrak{q}_{\beta}(P'_{\alpha}(r')\varphi)(m)\\
		&+\lambda_{\beta}\mathfrak{q}_{\alpha}(r'\varphi)(m)+\lambda_{\alpha}\mathfrak{q}_{\beta}(r'\varphi)(m)\\
		=&(\mathfrak{q}_{\alpha}(r'\mathfrak{q}_{\beta}(\varphi))+\mathfrak{q}_{\beta}(P'_{\alpha}(r')\varphi)+\lambda_{\beta}\mathfrak{q}_{\alpha}(r'\varphi)+\lambda_{\alpha}\mathfrak{q}_{\beta}(r'\varphi))(m).
	\end{align*}
	Hence
	\begin{equation*}
		P'_{\alpha}(r')\mathfrak{q}_{\beta}(\varphi)=\mathfrak{q}_{\alpha}(r'\mathfrak{q}_{\beta}(\varphi))+\mathfrak{q}_{\beta}(P'_{\alpha}(r')\varphi)+\lambda_{\beta}\mathfrak{q}_{\alpha}(r'\varphi)+\lambda_{\alpha}\mathfrak{q}_{\beta}(r'\varphi),\quad \forall\ r'\in R',
	\end{equation*}
	as required.
	
	(d). The verification proceeds similarly to Item (c).
\end{proof}

\subsection{Projective and injective multiple Rota-Baxter module}
By \cite{Cw}, in any abelian category with enough projective and injective objects, the derived functors of $\mathrm{Hom}$ can be constructed via both projective and injective resolutions. In this subsection, we show that there are enough projective and injective objects in $_{(R,P_{\Omega})}{\bf Mod}$.

\begin{defn}
	A left $(R,P_{\Omega})$-module $(S,\mathfrak{s}_{\Omega})$ is {\bf projective} if, for any left $(R,P_{\Omega})$-module epimorphism $\theta:(M,\mathfrak{m}_{\Omega})\longrightarrow(N,\mathfrak{n}_{\Omega})$ and any left $(R,P_{\Omega})$-module homomorphism $\varphi:(S,\mathfrak{s}_{\Omega})\longrightarrow(N,\mathfrak{n}_{\Omega})$, there exists a left $(R,P_{\Omega})$-module homomorphism $\overline{\varphi}:(S,\mathfrak{s}_{\Omega})\longrightarrow(M,\mathfrak{m}_{\Omega})$ such that $\varphi=\overline{\varphi}\circ \theta$. This is represented by the commutative diagram:
	\begin{displaymath}
		\xymatrix{
			& (S,\mathfrak{s}_{\Omega}) \ar@{-->}[dl]_-{\overline{\varphi}} \ar[d]^-{\varphi}  \\
			(M,\mathfrak{m}_{\Omega}) \ar[r]_-{\theta} & (N,\mathfrak{n}_{\Omega}) \ar[r]_-{} & 0.	\\
		}
	\end{displaymath}
\end{defn}

\begin{prop}{\label{E}}
	In the category of left $(R,P_{\Omega})$-modules, a free left $(R,P_{\Omega})$-module is projective.
\end{prop}

\begin{proof}
	The verification follows the same approach as in the case of trivial left modules. Let $(F(X),\mathfrak{p}_{\Omega})$ be the free left $(R, P_{\Omega})$-module generated by $X$, equipped with the embedding $j_X : X\longrightarrow F(X)$. Let $\theta:(M,\mathfrak{m}_{\Omega})\longrightarrow(N,\mathfrak{n}_{\Omega})$ be a left $(R, P_{\Omega})$-module epimorphism and let
	\[ \varphi:(F(X),\mathfrak{p}_{\Omega})\longrightarrow (N,\mathfrak{n}_{\Omega}), \]
	be a left $(R,P_{\Omega})$-module homomorphism. By surjectivity of $\theta$, for any $x\in X$, there exists a $m_x\in M$ with $\theta(m_x)=(\varphi\circ j_X)(x)$. Hence we have a set map $\varphi':X\longrightarrow M$ defined by $x\longmapsto m_x$, and $\theta\circ \varphi'=\varphi\circ j_X$. By the universal property of $F(X)$, there exists a left $(R, P_{\Omega})$-module homomorphism $\overline{\varphi}:F(X)\longrightarrow M$, such that $\overline{\varphi}\circ j_X=\varphi'$. Then we have $\theta\circ\overline{\varphi}\circ j_X=\theta\circ \varphi'=\varphi \circ j_X$, by the universal property of $F(X)$ again we get  $\theta\circ\overline{\varphi}=\varphi$.
\end{proof}

By Proposition~\ref{BB} and Proposition~\ref{E}, there are enough projective objects in $\mathrm{_{(R,P_{\Omega})}\bf{Mod}}$.

We now proceed to define injective left $(R,P_{\Omega})$-modules and show that $\mathrm{_{(R,P_{\Omega})}\bf{Mod}}$ has enough injective objects.

\begin{defn}
	A left $(R,P_{\Omega})$-module $(T,\mathfrak{t}_{\Omega})$ is {\bf injective} if, for any left $(R,P_{\Omega})$-module monomorphism $\theta:(M,\mathfrak{m}_{\Omega})\longrightarrow(N,\mathfrak{n}_{\Omega})$ and any left $(R,P_{\Omega})$-module homomorphism $\varphi:(M,\mathfrak{m}_{\Omega}) \longrightarrow (T,\mathfrak{t}_{\Omega})$, there exists a left $(R,P_{\Omega})$-module homomorphism $\overline{\varphi}:(N,\mathfrak{n}_{\Omega})\longrightarrow(T,\mathfrak{t}_{\Omega})$ such that $\varphi=\overline{\varphi}\circ \theta$. This is represented by the commutative diagram:
	\begin{equation*}
		\xymatrix{
			&  (T,\mathfrak{t}_{\Omega}) & \\
			0\ar[r]^-{} &(M,\mathfrak{m}_{\Omega}) \ar[u]^-{\varphi} \ar[r]_-{\theta} & (N,\mathfrak{n}_{\Omega}). \ar@{-->}[ul]_-{\overline{\varphi}}
		}
	\end{equation*}
\end{defn}

% \begin{exam}
	% 	If $(T,\mathfrak{t}_{\Omega})$ is a left $(R,P_{\Omega})$-module and $(I,I_{\Omega})$ is a left $(R,P_{\Omega})$-ideal, defines a left $(R,P_{\Omega})$-module homomorphism $\theta_t:I\longrightarrow T$ by $\theta_t(i)=it$ each $t\in T$. Then $(T,\mathfrak{t}_{\Omega})$ is injective when for all $(I,I_{\Omega})$ every $\theta:I\longrightarrow T$ is $\theta_t$ for some $i$.
	% \end{exam}

We introduce the concept of ring of multiple Rota-Baxter operators, based on the framework developed in \cite{LQ}.

\begin{defn}
	For the multiple Rota-Baxter algebra $(R,P_{\Omega})$, the {\bf ring of multiple Rota-Baxter operators on} $(R,P_{\Omega})$, denoted by $R_{MRB}\left\langle Q_{\Omega}\right\rangle$, is constructed as the quotient algebra
	\begin{equation*}
		R_{MRB}\left\langle Q_{\Omega}\right\rangle ={\bf k}\left\langle R,{\bf k}\left\langle Q_{\Omega}\right\rangle\right\rangle /I_{R,Q_{\Omega}},
	\end{equation*}
	where ${\bf k}\left\langle R,{\bf k}\left\langle Q_{\Omega}\right\rangle\right\rangle$ is the free product of $(R,P_{\Omega})$ and ${\bf k}\left\langle Q_{\Omega}\right\rangle$ with $Q_{\Omega}:=(Q_{\omega})_{\omega\in\Omega}$ a set of formal variables,
	$I_{R,Q_{\Omega}}$ is the ideal generated by the set
	\begin{equation*}
		\left\lbrace Q_{\alpha}rQ_{\beta}-P_{\alpha}(r)Q_{\beta}+Q_{\beta}P_{\alpha}(r)+\lambda_{\beta}Q_{\alpha}r+\lambda_{\alpha}Q_{\beta}r\ |\ r\in R,\alpha,\beta\in\Omega\right\rbrace .
	\end{equation*}
	The multiplicative identity of this ring is denoted by ${\bf 1}_{R_{MRB}\left\langle Q_{\Omega}\right\rangle }$.
\end{defn}

\begin{prop}{\label{F}}
	The category of left $(R,P_{\Omega})$-modules is isomorphism to the category of $R_{MRB}\left\langle Q_{\Omega}\right\rangle$-modules under the following correspondence.
	\begin{enumerate}
		\item If $(S,\mathfrak{s}_{\Omega})$ is a left $(R,P_{\Omega})$-module, then $R$-module $S$ is also $R_{MRB}\left\langle Q_{\Omega}\right\rangle$-module with the action $Q_{\omega}\cdot s:=\mathfrak{s}_{\omega}(s),s\in S$.
		\item If $S$ is a left $R_{MRB}\left\langle Q_{\Omega}\right\rangle$-module, then $(S,\mathfrak{s}_{\Omega})$ is also a left $(R,P_{\Omega})$-module with a set $\mathfrak{s}_{\Omega}:=(\mathfrak{s}_{\omega})_{\omega\in\Omega}$ of linear operators
		\begin{equation*}
			\mathfrak{s}_{\omega}:S\longrightarrow S,\quad \mathfrak{s}_{\omega}(s):=Q_{\omega}\cdot s,\quad s\in S.
		\end{equation*}
	\end{enumerate}
	
	\begin{proof}
		(a). Suppose that $(S,\mathfrak{s}_{\Omega})$ is a left $(R,P_{\Omega})$-module, the $R$-module structure on $S$ defines an algebra homomorphism $\varphi:R\longrightarrow \mathrm{End}_{\bf k}(S)$ with $r\cdot s= rs$ for $r\in R$. The linear maps $\mathfrak{s}_{\omega}\in \mathrm{End}_{\bf k}(S)$ define an algebra homomorphism $\psi:{\bf k}\left\langle Q_{\Omega}\right\rangle\longrightarrow \mathrm{End}_{\bf k}(S)$ with $\psi(Q_{\omega})=\mathfrak{s}_{\omega}$. By the universal property of free product, there exists a unique algebra homomorphism $\theta:{\bf k}\left\langle R,{\bf k}\left\langle Q_{\Omega}\right\rangle\right\rangle\longrightarrow \mathrm{End}_{\bf k}(S)$. For $\alpha,\beta\in\Omega, s\in S,r\in R$, we have
		\begin{equation*}
			\begin{split}
				&\theta(Q_{\alpha}rQ_{\beta}-P_{\alpha}(r)Q_{\beta}+Q_{\beta}P_{\alpha}(r)+\lambda_{\beta}Q_{\alpha}r+\lambda_{\alpha}Q_{\beta}r)(s)\\
				=&\theta(Q_{\alpha}rQ_{\beta})(s)-\theta(P_{\alpha}(r)Q_{\beta})(s)+\theta(Q_{\beta}P_{\alpha}(r))(s)+\theta(\lambda_{\beta}Q_{\alpha}r)(s)+\theta(\lambda_{\alpha}Q_{\beta}r)(s)\\
				=&\mathfrak{s}_{\alpha}(r\mathfrak{s}_{\beta}(s))-P_{\alpha}(r)\mathfrak{s}_{\beta}(s)+\mathfrak{s}_{\alpha}(P_\alpha(r)s)+\lambda_{\beta}\mathfrak{s}_{\alpha}(rs)+\lambda_{\alpha}\mathfrak{s}_{\beta}(rs)\\
				=&0.
			\end{split}
		\end{equation*}
		Since $I_{R,Q_{\Omega}}$ is the ideal generated by elements of the form
		\begin{equation*}
			Q_{\alpha}rQ_{\beta}-P_{\alpha}(r)Q_{\beta}+Q_{\beta}P_{\alpha}(r)+\lambda_{\beta}Q_{\alpha}r+\lambda_{\alpha}Q_{\beta}r,
		\end{equation*}
		$\theta(I_{R,Q_{\Omega}})=0$. Hence $\theta$ induces an algebra homomorphism $\overline{\varphi}:R_{MRB}\left\langle Q_{\Omega}\right\rangle\longrightarrow \mathrm{End}_{\bf k}(S)$ that defines an $R_{MRB}\left\langle Q_{\Omega}\right\rangle$-module structure on $S$.
		
		(b). For any $R_{MRB}\left\langle Q_{\Omega}\right\rangle$-module $S$, the algebra homomorphism $\theta:R_{MRB}\left\langle Q_{\Omega}\right\rangle\longrightarrow \mathrm{End}_{\bf k}(S)$ restricted to the subalgebra $R$ gives an $R$-module structure on $S$.
		For each $\omega\in\Omega$, defines $\mathfrak{s}_{\omega}(s):=Q_{\omega}\cdot s$. For $r\in R$, $s\in S$ and $\alpha,\beta\in\Omega$,
		\begin{align*}
			P_{\alpha}(r)\mathfrak{s}_{\beta}(s)=&P_{\alpha}(r)Q_{\beta}(s)=(P_{\alpha}(r)Q_{\beta})s\\
			=&(Q_{\alpha}rQ_{\beta}+Q_{\beta}P_{\alpha}(r)+\lambda_{\beta}Q_{\alpha}r+\lambda_{\alpha}Q_{\beta}r)s\\
			=&Q_{\alpha}rQ_{\beta}s+Q_{\beta}P_{\alpha}(r)s+\lambda_{\beta}Q_{\alpha}rs+\lambda_{\alpha}Q_{\beta}rs\\
			=&Q_{\alpha}(rQ_{\beta}s)+Q_{\beta}(P_{\alpha}(r)s)+\lambda_{\beta}Q_{\alpha}(rs)+\lambda_{\alpha}Q_{\beta}(rs)\\
			=&\mathfrak{s}_{\alpha}(rQ_{\beta}s)+\mathfrak{s}_{\beta}(P_{\alpha}(r)s)+\lambda_{\beta}\mathfrak{s}_{\alpha}(rs)+\lambda_{\alpha}\mathfrak{s}_{\beta}(rs)\\
			=&\mathfrak{s}_{\alpha}(r\mathfrak{s}_{\beta}(s))+\mathfrak{s}_{\beta}(P_{\alpha}(r)s)+\lambda_{\beta}\mathfrak{s}_{\alpha}(rs)+\lambda_{\alpha}\mathfrak{s}_{\beta}(rs).
		\end{align*}
		Hence $(S,\mathfrak{s}_{\omega}(s))$ is a left $(R,P_{\Omega})$-module.
	\end{proof}

	In particular, left $(R, P_{\Omega})$-ideals of $R_{MRB}\left\langle Q_{\Omega}\right\rangle$ are of the form $(I,\overline{P}_{\Omega}|_I)$ where $I$ is a left ideal of $R_{MRB}\left\langle Q_{\Omega}\right\rangle$ and $\overline{P}_{\Omega}=(\overline{P}_{\omega})_{\omega\in\Omega}:R_{MRB}\left\langle Q_{\Omega}\right\rangle\longrightarrow R_{MRB}\left\langle Q_{\Omega}\right\rangle$ is the left multiplication by $Q_{\omega}$ , for $\omega\in\Omega$.
	% Furthermore, the left $(R,P_{\Omega})$-ideals $(I,\overline{P}_{\Omega}|_{I})$ of $R_{MRB}\left\langle Q\right\rangle$ consists of a left ideal $I$ of $R_{MRB}\left\langle Q\right\rangle$ and the set of the restricted operators $\overline{P}_{\omega}|_{I}:I\longrightarrow I$, where each $\overline{P}_{\omega}$ acts as left multiplication by $Q$.
\end{prop}

Now we generalize the Baer Criterion to the case of $_{(R,P_{\Omega})}\bf{Mod}$.

\begin{prop}{\label{G}}
	Let $(M,\mathfrak{m}_{\Omega})$ be a left $(R,P_{\Omega})$-module, Then $(M,\mathfrak{m}_{\Omega})$ is injective in $_{(R,P_{\Omega})}\bf{Mod}$ if and only if, for every left $(R,P_{\Omega})$-ideal $(I,\overline{P}_{\Omega}|_I)$ of $(R_{MRB}\left\langle Q_{\Omega}\right\rangle,\overline{P}_{\Omega})$, every left $(R,P_{\Omega})$-module homomorphism $\varphi:(I,\overline{P}_{\Omega}|_I)\longrightarrow (M,\mathfrak{m}_{\Omega})$ can be extended to $\overline{\varphi}:(R_{MRB}\left\langle Q_{\Omega}\right\rangle,\overline{P}_{\Omega})\longrightarrow (M,\mathfrak{m}_{\Omega})$.
\end{prop}

\begin{proof}
	We refer to the proof of the Baer Criterion in \cite{R}.
	
	Assume that $(M,\mathfrak{m}_{\Omega})$ is an injective left $(R,P_{\Omega})$-module, then every left $(R,P_{\Omega})$-module homomorphism $\varphi: (I,\overline{P}_{\Omega}|_I)\longrightarrow (M,\mathfrak{m}_{\Omega})$ can be extended to $\overline{\varphi}:(R_{MRB}\left\langle Q_{\Omega}\right\rangle,\overline{P}_{\Omega})\longrightarrow (M,\mathfrak{m}_{\Omega})$ by the injectivity of $(M,\mathfrak{m}_{\Omega})$.
	
	Conversely, suppose that for any left $(R,P_{\Omega})$-ideal $(I,\overline{P}_{\Omega}|_I)$ of $(R_{MRB}\left\langle Q_{\Omega}\right\rangle,\overline{P}_{\Omega})$, any left $(R,P_{\Omega})$-module homomorphism $\varphi: (I,\overline{P}_{\Omega}|_I)\longrightarrow (M,\mathfrak{m}_{\Omega})$ can extend to $\overline{\varphi}:(R_{MRB}\left\langle Q_{\Omega}\right\rangle,\overline{P}_{\Omega})\longrightarrow (M,\mathfrak{m}_{\Omega})$. We prove $(M,\mathfrak{m}_{\Omega})$ is injective.
	
	Let $(T,\mathfrak{t}_{\Omega})$ be a left $(R,P_{\Omega})$-module, $(S,\mathfrak{s}_{\Omega})$ be a left $(R,P_{\Omega})$-submodule of $(T,\mathfrak{t}_{\Omega})$ and $f$ be a left $(R,P_{\Omega})$-module monomorphism from $(S,\mathfrak{s}_{\Omega})$ to $(T,\mathfrak{t}_{\Omega})$. Suppose that $\varphi':(S,\mathfrak{s}_{\Omega})\longrightarrow (M,\mathfrak{m}_{\Omega})$ is a left $(R,P_{\Omega})$-module homomorphism. Define a set of left $(R,P_{\Omega})$-module homomorphisms
	\begin{equation*}
		F:=\{\psi: (V,\mathfrak{v}_{\Omega})\longrightarrow(M,\mathfrak{m}_{\Omega})|(S,\mathfrak{s}_{\Omega})\leqslant(V,\mathfrak{v}_{\Omega})\leqslant (T,\mathfrak{t}_{\Omega}),\quad \psi|_{(S,\mathfrak{s}_{\Omega})}=\varphi'\}.
	\end{equation*}
	Apparently $\varphi'\in F$, then $F$ is a non-empty set and partially order by the domain $(V,\mathfrak{v}_{\Omega})$. By Zorn's lemma, $F$ has a maximal element $\psi_{max}:(V,\mathfrak{v}_{\Omega})\longrightarrow (M,\mathfrak{m}_{\Omega})$.
	
	If $V=T$, then we are done.
	
	If $V\neq T$, we choose $a\in T\backslash V$ and define
	\begin{equation*}
		\overline{V}:=\{r\in R_{MRB}\left\langle Q_{\Omega}\right\rangle|ra\in V\}.
	\end{equation*}
	% We prove that $\overline{V}$ is a left $(R,P_{\Omega})$-ideal of $R_{MRB}\left\langle Q_{\Omega}\right\rangle$.
	For each $\overline{r}\in \overline{V}\in R_{MRB}\left\langle Q_{\Omega}\right\rangle$ and $\omega\in\Omega$,
	\begin{equation*}
		(\overline{P}_{\omega}(\overline{r}))a=(Q_{\omega}\cdot \overline{r})a=Q_{\omega}\cdot (\overline{r}a).
	\end{equation*}
	By Item (a) of Proposition \ref{F}, $Q_{\omega}\cdot (\overline{r}a)\in \overline{V}$. Then $\overline{P}_{\omega}(\overline{V})\subseteq\overline{V}$ and $\overline{V}$ is a left $(R,P_{\Omega})$-ideal of $R_{MRB}\left\langle Q_{\Omega}\right\rangle$.
	Define a map
	\begin{equation*}
		\theta: \overline{V}\longrightarrow M,\quad r\longmapsto \psi(ra).
	\end{equation*}
	For $r_1\in R$ and $\overline{r}\in \overline{V}$,
	\begin{equation*}
		\theta(r_1\overline{r})=\psi(r_1\overline{r}a)=r_1\psi(\overline{r}a)=r_1\theta(\overline{r}).
	\end{equation*}
	Then $\theta$ is a left $(R,P_{\Omega})$-module homomorphism.
	
	By assumption, $\theta$ extends to $\theta':(R_{MRB}\left\langle Q_{\Omega}\right\rangle,\overline{P}_{\Omega})\longrightarrow (M,\mathfrak{m}_{\Omega})$ with $\theta'(r)=\psi(ra)$. Let $x:=\theta'({\bf 1}_{R_{MRB}\left\langle Q_{\Omega}\right\rangle})$ and define
	\begin{equation*}
		\eta: V+R_{MRB}\left\langle Q_{\Omega}\right\rangle a\longrightarrow M,\quad v+ra\longmapsto \psi(v)+rx,\quad \forall v\in V, r\in R_{MRB}\left\langle Q_{\Omega}\right\rangle.
	\end{equation*}
	% Since $\overline{P}_{\omega}(v+ra)=Q_{\omega}\cdot (v+ra)=Q_{\omega}\cdot v+Q_{\omega}\cdot(ra)\in V+{R_{MRB}\left\langle Q_{\Omega}\right\rangle}a$, then $\overline{P}_{\omega}(V+R_{MRB}\left\langle Q_{\Omega}\right\rangle a)\subseteq V+R_{MRB}\left\langle Q_{\Omega}\right\rangle a$ and $V+R_{MRB}\left\langle Q_{\Omega}\right\rangle a$ is a left $(R,P_{\Omega})$-ideal of $R_{MRB}\left\langle Q_{\Omega}\right\rangle$.
	
	If $v+ra=v'+r'a$ for $v,v'\in V$ and $r,r'\in R_{MRB}\left\langle Q_{\Omega}\right\rangle$, we have
	\begin{equation*}
		\psi(v-v')=\psi((r'-r)a)=\theta'(r'-r)=(r'-r)x,
	\end{equation*}
	and $\psi(v+ra)=\psi(v'+r'a)$. Then $\eta$ is well-defined. For any $v+ra,v'+r'a\in V+R_{MRB}\left\langle Q_{\Omega}\right\rangle$, since
	\begin{equation*}
		\begin{split}
			\eta((v+ra)+(v'+r'a))&=\eta((v+v')+(ra+r'a))=\psi(v+v')+rx+r'x\\
			&=\psi(v)+\psi(v')+rx+r'x=\eta(v+ra)+\eta(v'+r'a),
		\end{split}
	\end{equation*}
	and
	\begin{equation*}
		r'\eta(v+ra)=r'(\psi(v)+rx)=\psi(r'v)+r'rx=
		\psi(r'v+r'rx)=\psi(r'(v+ra)),
	\end{equation*}
	$\eta$ is left $R_{MRB}\left\langle Q_{\Omega}\right\rangle$-module homomorphism. By Proposition \ref{F}, hence $\eta$ is a left $(R,P_{\Omega})$-module homomorphism such that $\eta|_{(S,\mathfrak{s}_{\Omega})}=\varphi'$. This contradicts the maximality of $\psi_{max}$. Hence $V=T$, as required.
\end{proof}

An abelian group $G$ is called divisible when for all $g\in G$ and nonzero $n\in\mathbb{Z}$, there exists $x\in G$ satisfying $g=nx$.

For $f\in\mathrm{Hom}_{\mathbb{Z}}(R_{MRB}\left\langle Q_{\Omega}\right\rangle,G)$ and $\omega\in\Omega$, define a set $\mathfrak{q}_{\Omega}:=(\mathfrak{q}_{\omega})_{\omega\in\Omega}$ of linear operators $\mathfrak{q}_{\omega}(f)=fQ_{\omega}$. For $r\in R$, $a\in R_{MRB}\left\langle Q_{\Omega}\right\rangle$ and $\alpha,\beta\in\Omega$, we have
\begin{equation*}
	\begin{split}
		P_{\alpha}(r)\mathfrak{q}_{\beta}(f)
		=&P_{\alpha}(r)(fQ_{\beta})=f(P_{\alpha}(r)Q_{\beta})\\
		=&f(Q_{\alpha}rQ_{\beta}+Q_{\beta}P_{\alpha}(r)+\lambda_{\beta}Q_{\alpha}r+\lambda_{\alpha}Q_{\beta}r)\\
		=&f(Q_{\alpha}rQ_{\beta})+f(Q_{\beta}P_{\alpha}(r))+f(\lambda_{\beta}Q_{\alpha}r)+f(\lambda_{\alpha}Q_{\beta}r)\\
		=&\mathfrak{q}_{\beta}(f)(Q_{\alpha}r)+\mathfrak{q}_{\beta}(P_{\alpha}(r)f)+\lambda_{\beta}\mathfrak{q}_{\alpha}(rf)+\lambda_{\alpha}\mathfrak{q}_{\beta}(rf)\\
		=&\mathfrak{q}_{\alpha}(r\mathfrak{q}_{\beta}(f))+\mathfrak{q}_{\beta}(P_{\alpha}(r)f)+\lambda_{\beta}\mathfrak{q}_{\alpha}(rf)+\lambda_{\alpha}\mathfrak{q}_{\beta}(rf).
	\end{split}
\end{equation*}
Hence $(\mathrm{Hom}_{\mathbb{Z}}(R_{MRB}\left\langle Q_{\Omega}\right\rangle,G),\mathfrak{q}_{\Omega})$ is a left $(R,P_{\Omega})$-module.

\begin{prop}{\label{H}}
	With the above notations, the pair $(\mathrm{Hom}_{\mathbb{Z}}(R_{MRB}\left\langle Q_{\Omega}\right\rangle,G),\mathfrak{q}_{\Omega})$ is an injective left $(R,P_{\Omega})$-module.
	% Since $a\in R_{MRB}\left\langle Q_{\Omega}\right\rangle$, $\mathfrak{q}_{\omega}(f)(a)=fQ_{\omega}(a)=f(Q_{\omega}a)\in G$, then $\mathfrak{q}_{\omega}(f)\in \mathrm{Hom}_{\mathbb{Z}}(R_{MRB}\left\langle Q_{\Omega}\right\rangle,G)$.
\end{prop}
\begin{proof}
	By Proposition \ref{G}, we need to prove that for any left ideal $I$ of $R_{MRB}\left\langle Q_{\Omega}\right\rangle$ with the embedding $\varphi: (I,\overline{P}_{\Omega}|_I)\longrightarrow (R_{MRB}\left\langle Q_{\Omega}\right\rangle,\overline{P}_{\Omega})$, and for any left $(R,P_{\Omega})$-module homomorphism
	\begin{equation*}
		\psi:(I,\overline{P}_{\Omega}|_I)\longrightarrow(\mathrm{Hom}_{\mathbb{Z}}(R_{MRB}\left\langle Q_{\Omega}\right\rangle,G),\mathfrak{q}_{\Omega}),
	\end{equation*}
	there is a left $(R,P_{\Omega})$-module homomorphism $\theta':(R_{MRB}\left\langle Q_{\Omega}\right\rangle,\overline{P}_{\Omega})\longrightarrow (\mathrm{Hom}_{\mathbb{Z}}(R_{MRB}\left\langle Q_{\Omega}\right\rangle,G),\mathfrak{q}_{\Omega})$ such that the following diagram is commutative
	\begin{equation*}
		\xymatrix{
			&  (\mathrm{Hom}_{\mathbb{Z}}(R_{MRB}\left\langle Q_{\Omega}\right\rangle,G),\mathfrak{q}_{\Omega}) & \\
			0\ar[r]^-{} &(I,\overline{P}_{\Omega}|_I) \ar[u]^-{\psi} \ar[r]_-{\varphi} & (R_{MRB}\left\langle Q_{\Omega}\right\rangle,\overline{P}_{\Omega}). \ar@{-->}[ul]_-{\theta'}
		}
	\end{equation*}
	
	Define a map $\eta : I\longrightarrow G$ by $\eta(i)=\psi(i)({\bf 1}_{R_{MRB}\left\langle Q_{\Omega}\right\rangle})$ for $i\in I$. For $z\in \mathbb{Z}$, since
	\begin{equation*}
		z\eta(i)=z\psi(i)({\bf 1}_{R_{MRB}\left\langle Q_{\Omega}\right\rangle})=\psi(zi)({\bf 1}_{R_{MRB}\left\langle Q_{\Omega}\right\rangle})=\eta(zi),
	\end{equation*}
	$\eta$ is a $\mathbb{Z}$-module homomorphism.
	By the injectivity of divisible abelian group $G$ as $\mathbb{Z}$-modules \cite{R}, we have a $\mathbb{Z}$-linear extension $\theta:R_{MRB}\left\langle Q_{\Omega}\right\rangle\longrightarrow G$, such that $\eta=\theta\circ \varphi$.
	We define $\theta':R_{MRB}\left\langle Q_{\Omega}\right\rangle \longrightarrow\mathrm{Hom}_{\mathbb{Z}}(R_{MRB}\left\langle Q_{\Omega}\right\rangle,G)$ by $\theta'(a)(b)=\theta(ba)$ for $a,b\in R_{MRB}\left\langle Q_{\Omega}\right\rangle$.
	Since $\theta'(a)(b)=\theta(ba)\in G$, $\theta'(a)\in \mathrm{Hom}_{\mathbb{Z}}(R_{MRB}\left\langle Q_{\Omega}\right\rangle,G)$.
	
	For $r\in R$, $a,b\in R_{MRB}\left\langle Q_{\Omega}\right\rangle$, we have
	\begin{equation*}
		\theta'(ra)(b)=\theta(b(ra))=\theta((br)a)=\theta'(a)(br)=(\theta'(a)(b))r=(r\theta'(a))(b).
	\end{equation*}
	Then $\theta'$ is also an $R$-module homomorphism.
	For each $\omega\in \Omega$ and $i\in I$, we have
	\begin{equation*}
		\begin{split}
			((\theta'\circ\overline{P}_{\omega})(a))(b)&=\theta'(\overline{P}_{\omega}(a))(b)=\theta(b\overline{P}_{\omega}(a))=\theta(b(Q_{\omega}a))\\&
			=\theta((bQ_{\omega})a)=\theta'(a)(bQ_{\omega})=((\mathfrak{q}_{\omega}\circ \theta')(a))(b).
		\end{split}
	\end{equation*}
	Then $\theta':(R_{MRB}\left\langle Q_{\Omega}\right\rangle,\overline{P}_{\Omega})\longrightarrow(\mathrm{Hom}_{\mathbb{Z}}(R_{MRB}\left\langle Q_{\Omega}\right\rangle,G),\mathfrak{q}_{\Omega})$ is also a left $(R,P_{\Omega})$-module homomorphism.
	For $r'\in R$, we have
	\begin{equation*}
		((\theta'\circ\varphi)(i))(r')=\theta'(i)(r')=\theta(r'i)=\eta(r'i)=\psi(r'i)({\bf 1}_{R_{MRB}\left\langle Q_{\Omega}\right\rangle})=(r'\psi(i))({\bf 1}_{R_{MRB}\left\langle Q_{\Omega}\right\rangle})=\psi(i)r'.
	\end{equation*}
	For $r'=Q_{\omega}$, we have
	\begin{equation*}
		\begin{split}
			((\theta'\circ\varphi)(i))(Q_{\omega})& = \theta'(i)(Q_{\omega})=\theta(Q_{\omega}i)=\eta(Q_{\omega}i)=\psi(Q_{\omega}i)({\bf 1}_{R_{MRB}\left\langle Q_{\Omega}\right\rangle})\\
			& = ((\psi\circ\overline{P}_{\omega})(i))({\bf 1}_{R_{MRB}\left\langle Q_{\Omega}\right\rangle})=((q_{\omega}\circ \psi)(i))({\bf 1}_{R_{MRB}\left\langle Q_{\Omega}\right\rangle})\\
			& = (q_{\omega}(\psi(i)))({\bf 1}_{R_{MRB}\left\langle Q_{\Omega}\right\rangle})=\psi(i)(Q_{\omega}).
		\end{split}
	\end{equation*}
	Thus $\theta'\circ \varphi=\psi$. Hence $(\mathrm{Hom}_{\mathbb{Z}}(R_{MRB}\left\langle Q_{\Omega}\right\rangle,G),\mathfrak{q}_{\Omega})$ is an injective left $(R,P_{\Omega})$-module.
\end{proof}

\begin{theorem}
	Let $(M,\mathfrak{m}_{\Omega})$ be a left $(R,P_{\Omega})$-module. Then $(M,\mathfrak{m}_{\Omega})$ can be embedded into an injective $(R,P_{\Omega})$-module.
	% There exists an injective left $(R,P_{\Omega})$-module $(\mathrm{Hom}_{\mathbb{Z}}(R_{MRB}\left\langle Q_{\Omega}\right\rangle,G),\mathfrak{q}_{\Omega})$ and a left $(R,P_{\Omega})$-module monomorphism from $(M,\mathfrak{m}_{\Omega})$ to $(\mathrm{Hom}_{\mathbb{Z}}(R_{MRB}\left\langle Q_{\Omega}\right\rangle,G),\mathfrak{q}_{\Omega})$, where $\mathfrak{q}_{\omega}(f)=fQ_{\omega}$ for $f\in\mathrm{Hom}_{\mathbb{Z}}(R_{MRB}\left\langle Q_{\Omega}\right\rangle,G),\omega\in\Omega$.
\end{theorem}
\begin{proof}
	By Proposition \ref{F}, $(M,\mathfrak{m}_{\Omega})$ is a left $R_{MRB}\left\langle Q_{\Omega}\right\rangle$-module under the action
	\begin{equation*}
		R_{MRB}\left\langle Q_{\Omega}\right\rangle\times M\longrightarrow M,\ r\cdot m\longmapsto rm,\ Q_{\omega}\cdot m\longmapsto \mathfrak{m}_{\omega}(m),
	\end{equation*}
	for $r\in R$, $m\in M$ and $\omega\in\Omega$.
	
	The pair $(\mathrm{Hom}_{\mathbb{Z}}(R_{MRB}\left\langle Q_{\Omega}\right\rangle,M),\overline{\mathfrak{q}}_{\Omega})$ is a left $(R,P_{\Omega})$-module, which is analogous to the supplementary content preceding Proposition \ref{H}, for $f\in \mathrm{Hom}_{\mathbb{Z}}(R_{MRB}\left\langle Q_{\Omega}\right\rangle,M)$ and $\omega\in\Omega$. Define a map
	\begin{equation*}
		\phi:(M,\mathfrak{m}_{\Omega})\longrightarrow(\mathrm{Hom}_{\mathbb{Z}}(R_{MRB}\left\langle Q_{\Omega}\right\rangle,M),\overline{\mathfrak{q}}_{\Omega}),\quad m\longmapsto f_m,
	\end{equation*}
	where $f_m(a) = am$, for $a \in R_{MRB}\left\langle Q_{\Omega}\right\rangle$. This is well-defined since each $f_m$ is a $\mathbb{Z}$-module homomorphism.
	
	For any $r\in R$, $a\in R_{MRB}\left\langle Q_{\Omega}\right\rangle$ and $m\in M$, we have
	\begin{equation*}
		\begin{split}
			\phi(rm)(a)&=f_{rm}(a)=a(rm)=(ar)m\\
			&=f_m(ar)=rf_m(a)=(r\phi(m))(a).
		\end{split}
	\end{equation*}
	Then $\phi$ is a left $R$-module homomorphism.
	For $\omega\in\Omega$, $m\in M$ and $a\in R_{MRB}\left\langle Q_{\Omega}\right\rangle$,
	\begin{equation*}
		\begin{split}
			((\phi\circ \mathfrak{m}_{\omega})(m))(a)& = \phi(\mathfrak{m}_{\omega}(m))(a)=a(Q_{\omega}m)=(aQ_{\omega})m
			= f_m(aQ_{\omega})\\ & = \phi(m)(aQ_{\omega})=(\overline{\mathfrak{q}}_{\omega}(\phi(m)))(a)=(\overline{\mathfrak{q}}_{\omega}\circ \phi(m))(a).
		\end{split}
	\end{equation*}
	Hence $\phi$ is a left $(R,P_{\Omega})$-module homomorphism.
	We finally show that $\phi$ is a monomorphism, for any $m_1,m_2\in M$, if $f_{m_1}=f_{m_2}$, then $am_1=f_{m_1}(a)=f_{m_2}(a)=am_2$ for all $a\in R_{MRB}\left\langle Q_{\Omega}\right\rangle$. If $a={\bf 1}_{R_{MRB}\left\langle Q_{\Omega}\right\rangle}$, then $m_1=m_2$.
	
	Every abelian group is a subgroup of a divisible abelian group \cite{R}. Suppose $G$ is a divisible abelian group such that $M$ as an abelian group is a subgroup of $G$ and $\theta':M\longrightarrow G$ is the embedding map. We define
	\begin{equation*}
		\theta:(\mathrm{Hom}_{\mathbb{Z}}(R_{MRB}\left\langle Q_{\Omega}\right\rangle,M),\overline{\mathfrak{q}}_{\Omega})\longrightarrow (\mathrm{Hom}_{\mathbb{Z}}(R_{MRB}\left\langle Q_{\Omega}\right\rangle,G),\mathfrak{q}_{\Omega}),\quad \eta\longrightarrow \theta'\circ\eta,
	\end{equation*}
	for $\eta\in \mathrm{Hom}_{\mathbb{Z}}(R_{MRB}\left\langle Q_{\Omega}\right\rangle,M)$.
	
	For any $r\in R$ and $a\in R_{MRB}\left\langle Q_{\Omega}\right\rangle$, we have
	\begin{equation*}
		(\theta(r\eta))(a)=\theta'((r\eta)(a))=(\theta'\circ \eta)(ar)=r(\theta'\circ \eta)(a)=(r\theta(\eta))(a).
	\end{equation*}
	Then $\theta$ is also a left $R$-module homomorphism. For $\omega\in\Omega$,
	\begin{equation*}
		\begin{split}
			((\mathfrak{q}_{\omega}\circ \theta))(a)
			&=(\mathfrak{q}_{\omega}(\theta(\eta)))(a)=\theta(\eta)(aQ_{\omega})=(\theta'\circ \eta)(aQ_{\omega})\\
			&=(\theta'(\overline{\mathfrak{q}}_{\omega}(\eta)))(a)=((\theta\circ \overline{\mathfrak{q}}_{\omega}(\eta))(a).
		\end{split}
	\end{equation*}
	Hence $\theta$ is also a left $(R,P_{\Omega})$-module homomorphism.
	
	Since $\theta'$ is a monomorphism and $\theta(\eta)=\theta'\circ\eta$, $\theta$ is a monomorphism.
	By Proposition \ref{H}, $(\mathrm{Hom}_{\mathbb{Z}}(R_{MRB}\left\langle Q_{\Omega}\right\rangle,G),\mathfrak{q}_{\Omega}) $ is an injective left $(R,P_{\Omega})$-module, and the composition
	\begin{equation*}
		\theta\circ \phi :\ (M,\mathfrak{m}_{\Omega})\xrightarrow{\phi}(\mathrm{Hom}_{\mathbb{Z}}(R_{MRB}\left\langle Q_{\Omega}\right\rangle,M),\overline{\mathfrak{q}}_{\Omega})\xrightarrow{\theta}(\mathrm{Hom}_{\mathbb{Z}}(R_{MRB}\left\langle Q_{\Omega}\right\rangle,G),\mathfrak{q}_{\Omega}),
	\end{equation*}
	is the required left $(R,P_{\Omega})$-module monomorphism.
\end{proof}

\section{Flat multiple Rota-Baxter modules}{\label{4}}
In this section, we turn our attention to flat multiple Rota-Baxter modules and establish the tensor product construction between two multiple Rota-Baxter modules.

\subsection{Tensor product of multiple Rota-Baxter modules}
We first introduce the definition of tensor products of multiple Rota-Baxter modules.

\begin{defn}
	Let $G$ be an additive abelian group. Suppose that $(_{(R,P_{\Omega})}N,\mathfrak{n}_{\Omega})$ is a left $(R,P_{\Omega})$-module and $(M_{(R,P_{\Omega})},\mathfrak{m}_{\Omega})$ is a right $(R,P_{\Omega})$-module.
	\begin{enumerate}
		\item A map $\varphi:M\times N\longrightarrow G$ is called $(R,P_{\Omega})$-{\bf bilinear} if for all $m,m_1,m_2\in M$, $n,n_1,n_2\in N$ and $r\in R$, $\varphi$ satisfies the following identities:
		\begin{equation*}
			\begin{split}
				\varphi(m_1+m_2,n)&=\varphi(m_1,n)+\varphi(m_2,n),\\
				\varphi(m,n_1+n_2)&=\varphi(m,n_1)+\varphi(m,n_2),\\
				\varphi(mr,n)&=\varphi(m,rn),\\
				\varphi(\mathfrak{m}_{\omega}(m),n)&=\varphi(m,\mathfrak{n}_{\omega}(n)).
			\end{split}
		\end{equation*}
		
		\item The {\bf tensor product} $M\otimes_{(R,P_{\Omega})}N$ is an abelian group equipped with an $(R,P_{\Omega})$-bilinear map
		\begin{equation*}
			\zeta:M\times N\longrightarrow M\otimes_{(R,P_{\Omega})}N,
		\end{equation*}
		satisfying the universal property: for every $(R,P_{\Omega})$-bilinear map $\varphi:M\times N\longrightarrow G$, there exists a unique abelian group homomorphism $\widetilde{\varphi}:M\otimes_{(R,P_{\Omega})}N\longrightarrow G$ such that $\varphi=\widetilde{\varphi}\circ \zeta$. This is represented by the commutative diagram:
		\begin{equation*}
			\xymatrix{
				M\times N \ar[dr]_-{\varphi} \ar[rr]^-{\zeta} &  &  M\otimes_{(R,P_{\Omega})}N \ar@{-->}[dl]^-{\widetilde{\varphi}} \\
				& G.
			}
		\end{equation*}
	\end{enumerate}
\end{defn}

Now we construct the tensor product of multiple Rota-Baxter modules.

\begin{theorem}
	Let $(M_{(R,P_{\Omega})},\mathfrak{m}_{\Omega})$ be a right $(R,P_{\Omega})$-module and $(_{(R,P_{\Omega})}N,\mathfrak{n}_{\Omega})$ be a left $(R,P_{\Omega})$-module. Let $G$ be a free abelian group generated by $M\times N$, and $I$ denote the subgroup of $G$ generated by elements of the form
	\begin{align*}
		&(m_1+m_2,n)-(m_1,n)-(m_2,n),\\
		&(m,n_1+n_2)-(m,n_1)-(m,n_2),\\
		&(mr,n)-(m,rn),\\
		&(\mathfrak{m}_{\omega}(m),n)-(m,\mathfrak{n}_{\omega}(n)),
	\end{align*}
	for $m,m_1,m_2\in M,n,n_1,n_2\in N,r\in R$ and $\omega\in\Omega$.
	Then the quotient group $G/I$ equipped with the bilinear map $\zeta:M\times N\longrightarrow G\longrightarrow G/I$ is $M\otimes_{(R,P_{\Omega})}N$.
\end{theorem}
\begin{proof}
	Define a pure tensor $m\otimes_{(R,P_{\Omega})}n:=\zeta(m,n)$, $\forall(m,n)\in M\times N$. We need to show that $G/I$ has the universal property of tensor product.
	
	Suppose that $\varphi:M\times N\longrightarrow G'$ is an $(R,P_{\Omega})$-bilinear map, where $G'$ is an abelian group. There exists a unique abelian group homomorphism $\overline{\varphi}:G\longrightarrow G'$ by the universal property of free abelian group. Since $\varphi$ is bilinear, $\overline{\varphi}$ vanishes on $I$, and then $\overline{\varphi}$ induces
	\begin{equation*}
		\widetilde{\varphi}:G/I\longrightarrow G',\quad \widetilde{\varphi}(m\otimes_{(R,P_{\Omega})}n)=\overline{\varphi}(m,n).
	\end{equation*}
	Apparently $\widetilde{\varphi}$ is a well-defined abelian group homomorphism. And we have
	\begin{equation*}
		\varphi(m,n)=\overline{\varphi}(m,n)=\widetilde{\varphi}(m\otimes_{(R,P_{\Omega})}n)=\widetilde{\varphi}\circ\zeta(m,n).
	\end{equation*}
	For linear extension of pure tensors, we have
	\begin{equation*}
		\widetilde{\varphi}(\sum_{i}^{}m_i\otimes_{(R,P_{\Omega})}n_i)=\sum_{i}^{}\widetilde{\varphi}(m_i\otimes_{(R,P_{\Omega})}n_i)=\sum_{i}^{}\widetilde{\varphi}(\zeta(m_i.n_i))=\sum_{i}^{}\varphi(m_i.n_i).
	\end{equation*}
	Then uniqueness of $\widetilde{\varphi}$ is determined by $\varphi$. Hence $G/I$ is the tensor product $M\otimes_{(R,P_{\Omega})}N$.
\end{proof}
Denote the category of the abelian groups by {\bf Ab}.
\begin{prop}
	With the above notations.
	\begin{enumerate}
		\item Let $(M_{(R,P_{\Omega})},\mathfrak{m}_{\Omega})$ be a right $(R,P_{\Omega})$-module, then there exists an additive functor $F_M :\ _{(R,P_{\Omega})}{\bf Mod} \longrightarrow {\bf Ab}$ defined by
		\begin{equation*}
			F_M(S)=M\otimes_{(R,P_{\Omega})}S,\quad F_M(\theta)=id_M\otimes_{(R,P_{\Omega})}\theta,
		\end{equation*}
		where $(S,\mathfrak{s}_{\Omega}),(T,\mathfrak{t}_{\Omega})$ are left $(R,P_{\Omega})$-modules, and $\theta:(S,\mathfrak{s}_{\Omega})\longrightarrow (T,\mathfrak{t}_{\Omega})$ is a left $(R,P_{\Omega})$-module homomorphism.
		
		\item Let $(_{(R,P_{\Omega})}M,\mathfrak{m}_{\Omega})$ be a left $(R,P_{\Omega})$-module, then there exists an additive  functor $G_M:{\bf Mod}_{(R,P_{\Omega})} \longrightarrow {\bf Ab}$ defined by
		\begin{equation*}
			G_M(S)=S\otimes_{(R,P_{\Omega})}M,\quad G_M(\theta)=\theta\otimes_{(R,P_{\Omega})}id_M,
		\end{equation*}
		where $(S,\mathfrak{s}_{\Omega}),(T,\mathfrak{t}_{\Omega})$ are right $(R,P_{\Omega})$-modules, and $\theta:(S,\mathfrak{s}_{\Omega})\longrightarrow (T,\mathfrak{t}_{\Omega})$ is a right $(R,P_{\Omega})$-module homomorphism.
	\end{enumerate}
\end{prop}

\begin{proof}
	(a) Suppose that $(S,\mathfrak{s}_{\Omega}),(H,\mathfrak{h}_{\Omega})$ are left $(R,P_{\Omega})$-modules, and $\theta':(S,\mathfrak{s}_{\Omega})\longrightarrow(H,\mathfrak{h}_{\Omega})$ is a left $(R,P_{\Omega})$-module homomorphism. Then
	\begin{equation*}
		F_M(\theta\circ \theta')=id_M\otimes_{(R,P_{\Omega})}(\theta\circ \theta')=(id_M\otimes_{(R,P_{\Omega})}\theta)\circ(id_M\otimes_{(R,P_{\Omega})}\theta')=F_M(\theta)\circ F_M(\theta').
	\end{equation*}
	Thus $F_M$ is a functor with $F_M(id_S)=id_M\otimes_{(R,P_{\Omega})}id_S$. We need to verify
	\begin{align*}
		F_M(\theta+\iota)=F_M(\theta)+F_M(\iota),
	\end{align*}
	where $\theta,\iota:S\longrightarrow T$ are left $(R,P_{\Omega})$-module homomorphisms. For any $m\otimes_{(R,P_{\Omega})}s\in M\otimes_{(R,P_{\Omega})}S$, we have
	\begin{equation*}
		\begin{split}
			F_M(\theta+\iota)(m\otimes_{(R,P_{\Omega})}s)
			&=id_{M\otimes_{(R,P_{\Omega})}S}(\theta+\iota)(m\otimes_{(R,P_{\Omega})}s)\\
			&=m\otimes_{(R,P_{\Omega})}((\theta+\iota)s)\\
			&=m\otimes_{(R,P_{\Omega})}(\theta(s)+\iota(s))\\	
			&=m\otimes_{(R,P_{\Omega})}\theta(s)+
			m\otimes_{(R,P_{\Omega})}\iota(s)\\
			&=(F_M(\theta)+F_M(\iota))(m\otimes_{(R,P_{\Omega})}s),
		\end{split}
	\end{equation*}
	as required.
	
	(b)The verification proceeds similarly to Item (a).
\end{proof}

\begin{prop}{\label{I}}
	Let $(R,P_{\Omega})$ and $(R',P_{\Omega}')$ be multiple Rota-Baxter algebras of pair weight $(\lambda_{\Omega},\lambda_{\Omega})$.
	\begin{enumerate}
		\item If $(_{(R,P_{\Omega})}S,\mathfrak{s}_{\Omega}^R)$ is a left $(R,P_{\Omega})$-module and $(_{(R',P_{\Omega}')}M_{(R,P_{\Omega})},\mathfrak{m}_{\Omega}^{R'},\mathfrak{m}_{\Omega}^R)$ is an $(R',P_{\Omega}')$-$(R,P_{\Omega})$-bimodule, then the tensor product $(M\otimes_{(R,P_{\Omega})}S,\mathfrak{q}_{\Omega})$ is a left $(R',P_{\Omega}')$-module defined by
		\begin{equation*}
			\begin{split}
				&r'(m\otimes_{(R,P_{\Omega})}s):=(r'm)\otimes_{(R,P_{\Omega})}s,\\
				&\mathfrak{q}_{\omega}(m\otimes_{(R,P_{\Omega})}s):=\mathfrak{m}_{\omega}^{R'}(m)\otimes_{(R,P_{\Omega})} s,\quad
			\end{split}
		\end{equation*}
		where $r'\in R'$, $m\in M$, $s\in S$ and $\omega\in\Omega$.
		\item If $(_{(R,P_{\Omega})}S_{(R',P_{\Omega}')},\mathfrak{s}_{\Omega}^R,\mathfrak{s}_{\Omega}^{R'})$ is an $(R,P_{\Omega})$-$(R',P_{\Omega}')$-bimodule and $(M_{(R,P_{\Omega})},\mathfrak{m}_{\Omega}^R)$ is a right $(R,P_{\Omega})$-module, then the tensor product $(M\otimes_{(R,P_{\Omega})}S,\mathfrak{q}_{\Omega})$ is a right $(R',P_{\Omega}')$-module defined by
		\begin{equation*}
			\begin{split}
				&(m\otimes_{(R,P_{\Omega})}s)r':=m\otimes_{(R,P_{\Omega})}(sr'),\\
				&\mathfrak{q}_{\omega}(m\otimes_{(R,P_{\Omega})}s):=m\otimes_{(R,P_{\Omega})}\mathfrak{s}_{\omega}^{R'}(s),
			\end{split}
		\end{equation*}
		where $r'\in R'$, $m\in M$, $s\in S$ and $\omega\in\Omega$.
	\end{enumerate}
\end{prop}
\begin{proof}
	(a) We need to show that $(M\otimes_{(R,P_{\Omega})}S,\mathfrak{q}_{\Omega})$ is a left $(R',P_{\Omega}')$-module. For any $ r'\in R',s\in S,m\in M$ and $\alpha,\beta\in\Omega$, we have
	\begin{equation*}
		\begin{split}
			P'_{\alpha}(r')\mathfrak{q}_{\beta}(m\otimes_{(R,P_{\Omega})}s)=&P'_{\alpha}(r')(\mathfrak{m}_{\beta}^{R'}(m)\otimes_{(R,P_{\Omega})}s)\\
			=&(P'_{\alpha}(r')\mathfrak{m}_{\beta}^{R'}(m))\otimes_{(R,P_{\Omega})}s\\
			=&\mathfrak{m}_{\alpha}^{R'}(r'\mathfrak{m}_{\beta}^{R'}(m))\otimes_{(R,P_{\Omega})}s+\mathfrak{m}_{\beta}^{R'}(P'_{\alpha}(r')m)\otimes_{(R,P_{\Omega})}s\\
			&+\lambda_{\alpha}\mathfrak{m}_{\beta}^{R'}(r'm)\otimes_{(R,P_{\Omega})}s+\lambda_{\beta}\mathfrak{m}_{\alpha}^{R'}(r'm)\otimes_{(R,P_{\Omega})}s\\
			=&\mathfrak{q}_{\alpha}(r'\mathfrak{m}_{\beta}^{R'}(m)\otimes_{(R,P_{\Omega})}s)+\mathfrak{q}_{\beta}(P'_{\alpha}(r')m\otimes_{(R,P_{\Omega})}s)\\
			&+\lambda_{\alpha}\mathfrak{q}_{\beta}(r'm\otimes_{(R,P_{\Omega})}s)+\lambda_{\beta}\mathfrak{q}_{\alpha}(r'm\otimes_{(R,P_{\Omega})}s)\\
			=&\mathfrak{q}_{\alpha}(r'(\mathfrak{m}_{\beta}^{R'}(m)\otimes_{(R,P_{\Omega})}s))+\mathfrak{q}_{\beta}(P'_{\alpha}(r')(m\otimes_{(R,P_{\Omega})}s))\\
			&+\lambda_{\alpha}\mathfrak{q}_{\beta}(r'm\otimes_{(R,P_{\Omega})}s)+\lambda_{\beta}\mathfrak{q}_{\alpha}(r'm\otimes_{(R,P_{\Omega})}s)\\
			=&\mathfrak{q}_{\alpha}(r'(\mathfrak{q}_{\beta}(m\otimes_{(R,P_{\Omega})}s)))+\mathfrak{q}_{\beta}(P'_{\alpha}(r')(m\otimes_{(R,P_{\Omega})}s))\\
			&+\lambda_{\alpha}\mathfrak{q}_{\beta}(r'm\otimes_{(R,P_{\Omega})}s)+\lambda_{\beta}\mathfrak{q}_{\alpha}(r'm\otimes_{(R,P_{\Omega})}s),
		\end{split}
	\end{equation*}
	as required.
	
	(b) The verification proceeds similarly to Item (a).
\end{proof}

\begin{theorem}
	Let $(R,P_{\Omega})$ and $(R',P_{\Omega}')$ be multiple Rota-Baxter algebras of pair weight $(\lambda_{\Omega},\lambda_{\Omega})$. We suppose that $(M_{(R,P_{\Omega})},\mathfrak{m}_{\Omega}^R)$ is a right $(R,P_{\Omega})$-module, $(_{(R,P_{\Omega})}S_{(R',P_{\Omega}')},\mathfrak{s}_{\Omega}^R,\mathfrak{s}_{\Omega}^{R'})$ is an $(R,P_{\Omega})$-$(R',P_{\Omega}')$-bimodule and $(T_{(R',P_{\Omega}')},\mathfrak{t}_{\Omega}^{R'})$ is a right $(R',P_{\Omega}')$-module. Then
	\begin{equation*}
		\mathrm{Hom}_{(R',P_{\Omega}')}(M\otimes_{(R,P_{\Omega})}S,T)\cong\mathrm{Hom}_{(R,P_{\Omega})}(M,\mathrm{Hom}_{(R',P_{\Omega}')}(S,T)).
	\end{equation*}
\end{theorem}
\begin{proof}
	For $m\otimes_{(R,P_{\Omega})}s\in M\otimes_{(R,P_{\Omega})}S$, we define two abelian group homomorphisms
	\begin{equation*}
		\theta : \mathrm{Hom}_{(R',P_{\Omega}')}(M\otimes_{(R,P_{\Omega})}S,T) \longrightarrow \mathrm{Hom}_{(R,P_{\Omega})}(M,\mathrm{Hom}_{(R',P_{\Omega}')}(S,T)),\ \theta(f)(m)(s)=f(m\otimes_{(R,P_{\Omega})}s),
	\end{equation*}
	and
	\begin{equation*}
		\theta' : \mathrm{Hom}_{(R,P_{\Omega})}(M,\mathrm{Hom}_{(R',P_{\Omega}')}(S,T)) \longrightarrow \mathrm{Hom}_{(R',P_{\Omega}')}(M\otimes_{(R,P_{\Omega})}S,T),\ \theta'(g)(m\otimes_{(R,P_{\Omega})}s)=g(m)(s),
	\end{equation*}
	where $f\in \mathrm{Hom}_{(R',P_{\Omega}')}(M\otimes_{(R,P_{\Omega})}S,T)$, and $g\in \mathrm{Hom}_{(R,P_{\Omega})}(M,\mathrm{Hom}_{(R',P_{\Omega}')}(S,T))$.
	Since
	\begin{equation*}
		\theta\circ\theta'(g)(m\otimes_{(R,P_{\Omega})}s)=\theta(g(m)(s))=g(m\otimes_{(R,P_{\Omega})}s),
	\end{equation*}
	\begin{equation*}
		\theta'\circ\theta(f)(m)(s)=\theta'(f(m\otimes_{(R,P_{\Omega})}s))=f(m)(s).
	\end{equation*}
	Hence $\theta,\theta'$ are mutually inverse isomorphisms, as required.
\end{proof}
The result demonstrates that $-\otimes_{(R,P_{\Omega})}S$ and $\mathrm{Hom}_{(R',P_{\Omega}')}(S,-)$ are adjoint functors.

\subsection{Flat multiple Rota-Baxter module}
Following the conventional approach, the right exactness of the multiple Rota-Baxter tensor product can be routinely proved. To explore its exactness, we give the notation of flatness for multiple Rota-Baxter modules.

\begin{defn}
	A right $(R,P_{\Omega})$-module $(M,\mathfrak{m}_{\Omega})$ is {\bf flat} if
	\begin{equation*}
		M \otimes_{(R,P_\Omega)} - :  _{(R,P_{\Omega})}{\bf Mod}\longrightarrow \mathbf{Ab}
	\end{equation*}
	is an exact tensor functor. More precisely, for every short exact sequence in the category of left $(R,P_\Omega)$-modules
	\begin{equation*}
		\xymatrix{0 \ar[r] & (S,\mathfrak{s}_{\Omega}) \ar[r]^-{i} & (N,\mathfrak{n}_{\Omega}) \ar[r]^-{j} & (T,\mathfrak{t}_{\Omega}) \ar[r] & 0},
	\end{equation*}
	the induced short sequence
	\begin{equation*}
		\xymatrix{0 \ar[r] & M\otimes_{(R,P_{\Omega})}S \ar[r]^-{id_M\otimes i} & M\otimes_{(R,P_{\Omega})}N \ar[r]^-{id_M\otimes j} & M\otimes_{(R,P_{\Omega})}T \ar[r] & 0}
	\end{equation*}
	is also exact.
\end{defn}

Thus a right $(R,P_{\Omega})$-module $(M,\mathfrak{m}_{\Omega})$ is flat if and only if $M\otimes -$ preserves injection.
% From the above results, we obtain that a right $(R,P_{\Omega})$-module $(M,\mathfrak{m}_{\Omega})$ is flat if and only if it preserves monomorphism, when $f:(S,\mathfrak{s}_{\Omega})\longrightarrow(T,\mathfrak{t}_{\Omega})$ is a left $(R,P_{\Omega})$-module monomorphism, then $id_M\otimes f:(M\otimes_{(R,P_{\Omega})}S,\overline{\mathfrak{s}}_{\Omega})\longrightarrow(M\otimes_{(R,P_{\Omega})}T,\overline{\mathfrak{t}}_{\Omega})$ is also a  group monomorphism.

\begin{theorem}{\label{J}}
	Let $(M,\mathfrak{m}_{\Omega})$ be a flat right $(R,P_{\Omega})$-module, then there exists a natural right $R$-module isomorphism $M\otimes_{(R,P_{\Omega})}R\cong M$.
\end{theorem}

\begin{proof}
	Let $\varphi: (R,P_{\Omega})\longrightarrow (R_{MRB}\left\langle Q_{\Omega}\right\rangle,\overline{P}_{\Omega})$ be an injective left $(R,P_{\Omega})$-module homomorphism, induced by the inclusion $R\longrightarrow R_{MRB}\left\langle Q_{\Omega}\right\rangle $. By the flatness of $(M,\mathfrak{m}_{\Omega})$,   $id_M\otimes_{(R,P_{\Omega})}\varphi: M\otimes_{(R,P_{\Omega})}R\longrightarrow M\otimes_{(R,P_{\Omega})}R_{MRB}\left\langle Q_{\Omega}\right\rangle$ is also an injective abelian group homomorphism. For any pure tensor $m\otimes_{R_{MRB}\left\langle Q_{\Omega}\right\rangle}x\in M\otimes_{R_{MRB}\left\langle Q_{\Omega}\right\rangle}R_{MRB}\left\langle Q_{\Omega}\right\rangle$ and $x\in {R_{MRB}\left\langle Q_{\Omega}\right\rangle}$, we have
	\begin{equation*}
		(id_M\otimes_{(R,P_{\Omega})}\varphi)(mx\otimes_{(R,P_{\Omega})}{\bf 1}_R)=mx\otimes_{(R,P_{\Omega})}{\bf 1}_{R_{MRB}\left\langle Q_{\Omega}\right\rangle}=m\otimes_{R_{MRB}\left\langle Q_{\Omega}\right\rangle}x.
	\end{equation*}
	Hence $id_M\otimes_{(R,P_{\Omega})}\varphi$ is surjective and is consequently an abelian group isomorphism.
	
	By Proposition \ref{I}, $M\otimes_{(R,P_{\Omega})}R$ and $M\otimes_{(R,P_{\Omega})}R_{MRB}\left\langle Q_{\Omega}\right\rangle$ are also right $R$-modules. For $m\otimes r_1\in M\otimes_{(R,P_{\Omega})}R$, $r_2\in R$, we have
	\begin{equation*}
		\begin{split}
			(id_M\otimes_{(R,P_{\Omega})}\varphi)((m\otimes r_1)r_2)=&(id_M\otimes_{(R,P_{\Omega})}\varphi)(m\otimes r_1r_2)=m\otimes_{(R,P_{\Omega})}\varphi(r_1r_2)\\
			=&m\otimes_{(R,P_{\Omega})}\varphi(r_1)r_2=((id_M\otimes_{(R,P_{\Omega})}\varphi)(m\otimes r_1))r_2.
		\end{split}
	\end{equation*}
	Then $id_M\otimes_{(R,P_{\Omega})}\varphi$ is a right $R$-module isomorphism.
	
	By Proposition \ref{F}, $(M,\mathfrak{m}_{\Omega})$ and $(R_{MRB}\left\langle Q_{\Omega}\right\rangle,\overline{P}_{\Omega})$ can be regard as $R_{MRB}\left\langle Q_{\Omega}\right\rangle$-modules.  We have $M\otimes_{(R,P_{\Omega})}R_{MRB}\left\langle Q_{\Omega}\right\rangle=M\otimes_{R_{MRB}\left\langle Q_{\Omega}\right\rangle}R_{MRB}\left\langle Q_{\Omega}\right\rangle\cong M$. Hence $M\otimes_{(R,P_{\Omega})}R\cong M\otimes_{(R,P_{\Omega})}R_{MRB}\left\langle Q_{\Omega}\right\rangle\cong M$.
\end{proof}

Let $\{(M_i, \mathfrak{m}_{\Omega_i}) \mid i \in I\}$ be a set of left $(R, P)$-modules. Then $\left(\bigoplus_{i \in I} M_i, \bigoplus_{i \in I} \mathfrak{m}_{\Omega_i}\right)$ is also a left $(R, P_{\Omega})$-module and is called the direct sum of $\{(M_i, \mathfrak{m}_{\Omega_i}) \mid i \in I\}$, where $\bigoplus_{i \in I} \mathfrak{m}_{\Omega_i}=\left(\bigoplus_{i \in I} \mathfrak{m}_{\omega_i}\right)_{\omega\in\Omega}$ is defined by
\begin{equation*}
	\bigoplus_{i \in I} \left(\mathfrak{m}_{\omega_i} \big( (m_i)_I \big)\right)_{\omega\in\Omega} = (\mathfrak{m}_{\omega_i}(m_i))_I,
\end{equation*}
for each $\omega\in\Omega$. Furthermore
\begin{equation*}
	\bigoplus_{i \in I} (M_i, \mathfrak{m}_{\Omega_i}) = \left(\bigoplus_{i \in I} M_i, \bigoplus_{i \in I} \mathfrak{m}_{\Omega_i}\right).
\end{equation*}
For each $i \in I$, the map $\phi_i : (M_i, \mathfrak{m}_{\Omega_i}) \longrightarrow \bigoplus_{i \in I} (M_i, \mathfrak{m}_{\Omega_i})$ is a monomorphism with $ \phi_i \circ \mathfrak{m}_{\Omega_i}=\left(\bigoplus_{i \in I} \mathfrak{m}_{\Omega_i}\right) \circ \phi_i$. The map $\psi_i : \bigoplus_{i \in I} (M_i, \mathfrak{m}_{\Omega_i}) \longrightarrow (M_i, \mathfrak{m}_{\Omega_i})$ is an epimorphism with $\psi_i \circ \left(\bigoplus_{i \in I} \mathfrak{m}_{\Omega_i}\right)=\mathfrak{m}_{\Omega_i} \circ \psi_i$. Further $\phi_i \circ \psi_i = id_{\bigoplus_{i \in I}(M_i, p_i)}$, and $\psi_i \circ \phi_i = id_{(M_i, \mathfrak{m}_{\Omega_i})}$.

\begin{lemma}{\label{K}}
	Let $\{(M_i,\mathfrak{m}_{\Omega_i})|i\in I\}, \{(N_i,\mathfrak{n}_{\Omega_i})|i\in I\}$ be sets of $(R,P_{\Omega})$-modules with corresponding $(R,P_{\Omega})$-module homomorphisms $\psi_i : (M_i,\mathfrak{m}_{\Omega_i})\longrightarrow(N_i,\mathfrak{n}_{\Omega_i})$ for each $i\in I$. Define the map
	\begin{equation*}
		\psi:=\bigoplus_{i\in I}\psi_i:\ \bigoplus_{i\in I}(M_i,\mathfrak{m}_{\Omega_i})\longrightarrow\bigoplus_{i\in I}(N_i,\mathfrak{n}_{\Omega_i}),\quad (m_i)_I\longmapsto(\psi_i(m_i))_I,
	\end{equation*}
	then $\psi$ is injective left $(R,P_{\Omega})$-module homomorphism if and only if each $\psi_i$ is also injective where $\psi_i:(M_i,\mathfrak{m}_{\Omega_i})\longrightarrow(N_i,\mathfrak{n}_{\Omega_i})$.
\end{lemma}

\begin{proof}
	The verification proceeds by $\text{ker}(\psi)=\oplus_{i\in I}\text{ker}(\psi_i)$.
\end{proof}

\begin{lemma}{\label{L}}
	Let $\{(M_i,\mathfrak{m}_{\Omega_i})|i\in I\}$ be a set of left $(R,P_{\Omega})$-modules and $(S,\mathfrak{s}_{\Omega})$ be a right $(R,P_{\Omega})$-module. Then
	\begin{equation*}
		S\otimes_{(R,P_{\Omega})}(\oplus_{i\in I}M_i)\cong\oplus_{i\in I}(S\otimes_{(R,P_{\Omega})}M_i).
	\end{equation*}
\end{lemma}

\begin{proof}
	We define group homomorphisms
	\begin{equation*}
		f_1 : S\otimes_{(R,P_{\Omega})}(\oplus_{i\in I}M_i)\longrightarrow\oplus_{i\in I}(S\otimes_{(R,P_{\Omega})}M_i),\quad f_1(s\otimes(m_i)_I)=(s\otimes m_i)_I,
	\end{equation*}
	and
	\begin{equation*}
		f_2 : \oplus_{i\in I}(S\otimes_{(R,P_{\Omega})}M_i)\longrightarrow S\otimes_{(R,P_{\Omega})}(\oplus_{i\in I}M_i),\quad f_2(s_i\otimes m_i)_I=(\prod_{i\in I}s_i)\otimes (m_i)_I.
	\end{equation*}
	It follows immediately that $f_1\circ f_2 = id_{\oplus_{i\in I}(S\otimes_{(R,P_{\Omega})}M_i)}$ and $f_2\circ f_1 = id_{S\otimes_{(R,P_{\Omega})}(\oplus_{i\in I}M_i)}$.
\end{proof}

\begin{prop}{\label{M}}
	Let $\{(M_i,\mathfrak{m}_{\Omega_i})|i\in I\}$ be a set of left $(R,P_{\Omega})$-modules. Then the $(R,P_{\Omega})$-module $\bigoplus_{i\in I}(M_i,\mathfrak{m}_{\Omega_i})$ is flat if and only if each $(R,P_{\Omega})$-module $(M_i,\mathfrak{m}_{\Omega_i})$ is flat.
\end{prop}

\begin{proof}
	Let $(S,\mathfrak{s}_{\Omega})$ and $(T,\mathfrak{t}_{\Omega})$ be right $(R,P_{\Omega})$-modules, and $\phi: (S,\mathfrak{s}_{\Omega})\longrightarrow(T,\mathfrak{t}_{\Omega})$ be a right $(R,P_{\Omega})$-module monomorphism.
	
	Assume that each left $(R,P_{\Omega})$-module $(M_i,\mathfrak{m}_{\Omega_i})$ is flat. Then each
	\begin{equation*}
		\phi\otimes id_{M_i} : S\otimes_{(R,P_{\Omega})}M_i\longrightarrow T\otimes_{(R,P_{\Omega})}M_i,
	\end{equation*}
	is also an injective abelian group homomorphism. By Lemma \ref{K}, the homomorphism
	\begin{equation*}
		\bigoplus_{i\in I}(\phi\otimes id_{M_i}):\bigoplus_{i\in I}(S\otimes_{(R,P_{\Omega})}M_i)\longrightarrow\bigoplus_{i\in I}(T\otimes_{(R,P_{\Omega})}M_i),
	\end{equation*}
	is injective. Thus the left $(R,P_{\Omega})$-module $\bigoplus_{i\in I}(M_i,\mathfrak{m}_{\Omega_i})$ is flat.
	%by Lemma \ref{L}.
	
	In contrast, assume that the left $(R,P_{\Omega})$-module $\bigoplus_{i\in I}(M_i,\mathfrak{m}_{\Omega_i})$ is flat. Then right $(R,P_{\Omega})$-module monomorphism $\phi$ induced the map
	\begin{equation*}
		\phi\otimes id_{(\oplus_{i\in I}M_i)}:S\otimes_{(R,P_{\Omega})}\left( \bigoplus_{i\in I}M_i\right) \longrightarrow T\otimes_{(R,P_{\Omega})}\left( \bigoplus_{i\in I}M_i\right),
	\end{equation*}
	is an injective group homomorphism. By Lemma \ref{L}, we have
	\begin{equation*}
		S\otimes_{(R,P_{\Omega})}\left( \bigoplus_{i\in I}M_i\right)\cong \bigoplus_{i\in I}(S\otimes_{(R,P_{\Omega})}M_i),\quad T\otimes_{(R,P_{\Omega})}\left( \bigoplus_{i\in I}M_i\right)\cong \bigoplus_{i\in I}(T\otimes_{(R,P_{\Omega})}M_i).
	\end{equation*}
	Then the group homomorphism
	\begin{equation*}
		\bigoplus_{i\in I}(S\otimes_{(R,P_{\Omega})}M_i)\longrightarrow\bigoplus_{i\in I}(T\otimes_{(R,P_{\Omega})}M_i),
	\end{equation*}
	is injective. By Lemma \ref{K}, the map
	\begin{equation*}
		\phi\otimes id_{M_i}: S\otimes_{(R,P_{\Omega})}M_i\longrightarrow T\otimes_{(R,P_{\Omega})}M_i,\quad \forall i\in I,
	\end{equation*}
	is injective. Thus each left $(R,P_{\Omega})$-module $(M_i,\mathfrak{m}_{\Omega_i})$ is flat.
\end{proof}

\begin{theorem}{\label{N}}
	Every free left $(R,P_{\Omega})$-module is flat.
\end{theorem}

\begin{proof}
	Let $(M_R(X)/I_X,\overline{\mathfrak{m}}_{\Omega}')$ be the free left $(R,P_{\Omega})$-module generated by $X$ as constructed in Theorem \ref{C}. Let $(N,\mathfrak{n}_{\Omega})\longrightarrow (N',\mathfrak{n}_{\Omega}')$ be right $(R,P_{\Omega})$-module homomorphism. We need to show that if $(N,\mathfrak{n}_{\Omega})\longrightarrow (N',\mathfrak{n}_{\Omega}')$ is a right $(R,P_{\Omega})$-module monomorphism, then the induced map $N\otimes_{(R,P_{\Omega})}(M_R(X)/I_X)\longrightarrow N'\otimes_{(R,P_{\Omega})}(M_R(X)/I_X)$ is also a abelian group monomorphism.
	
	First, let $X=\{x\}$ be a singleton and $(S,\mathfrak{s}_{\Omega})$ a right $(R,P_{\Omega})$-module, define the maps
	\begin{equation*}
		\begin{split}
			\varphi : S\otimes_{(R,P_{\Omega})}\left(\bigcup^{\infty}_{n\geqslant1}M_n(x)/I_{\left\{x\right\}}\right)&\longrightarrow S,\\
			s\otimes\left(r_1\otimes\omega_1\otimes\cdots\otimes r_n\otimes x+I_{\left\{x\right\}}\right)&\longmapsto sr_1\cdots r_n,
		\end{split}
	\end{equation*}
	and
	\begin{equation*}
		\begin{split}
			\varphi' : S&\longrightarrow S\otimes_{(R,P_{\Omega})}\left(\bigcup^{\infty}_{n\geqslant1}M_n(x)/I_{\left\{x\right\}}\right),\\
			s &\longmapsto s\otimes\left({\bf 1}_{R}\otimes\omega\otimes x+I_{\left\{x\right\}}\right),
		\end{split}
	\end{equation*}
	for $s\in S$, $r_1\otimes\omega_1\otimes\cdots\otimes r_n\otimes x\in M_n(x)$.
	
	% Since
	% \begin{equation*}
		%     \varphi\circ \varphi'(s)=\varphi\left(s\otimes\left({\bf 1}_{R}\otimes x+I_{\left\{x\right\}}\right)\right)=s,
		% \end{equation*}
	% \begin{align*}
		%     \varphi'\circ \varphi\left(s\otimes\left(r_1\otimes\omega_1\otimes\cdots\otimes r_n\otimes x+I_{\left\{x\right\}}\right)\right)=\varphi'(sr_1\cdots r_n)
		%     =s\otimes((r_1\otimes\cdots \otimes r_n)x+I_{\left\{x\right\}})
		% \end{align*}
	Apparently $\varphi\circ \varphi'=id_M$ and $\varphi'\circ \varphi=id_{S\otimes_{(R,P_{\Omega})}\left(\left(\bigoplus_{n\geqslant1}R^{\otimes n}\right)x/I_{\left\{x\right\}}\right)}$, then the maps $\varphi,\varphi'$ are mutually inverse isomorphisms and
	\begin{equation*}
		S\otimes_{(R,P_{\Omega})}(M_R(\left\{x\right\})/I_{\left\{x\right\}})\cong\ S.
	\end{equation*}
	For any set $X$, through $M_R(X)/I_X\cong \bigoplus_{x\in X}M_R(\left\{x\right\})/I_{\left\{x\right\}}$, we have
	\begin{equation*}
		\begin{split}
			S\otimes_{(R,P_{\Omega})}(M_R(X)/I_X)\ &\cong\ S\otimes_{(R,P_{\Omega})}\left(\bigoplus_{x\in X}M_R(\left\{x\right\})/I_{\left\{x\right\}}\right)\\
			&\cong\ \bigoplus_{x\in X}(S\otimes_{(R,P_{\Omega})}M_R(\left\{x\right\})/I_{\left\{x\right\}})\quad(\text{by Lemma \ref{K}}) \\
			&\cong\ \bigoplus_{x\in X}S.
		\end{split}
	\end{equation*}
	Applying the above results to right $(R,P_{\Omega})$-modules $N'$ and $N$, we have
	$N'\otimes_{(R,P_{\Omega})}(M_R(X)/I_X)\cong\oplus_{x\in X}N'$ and $N\otimes_{(R,P_{\Omega})}(M_R(X)/I_X)\cong\oplus_{x\in X}N$. The group homomorphism $\oplus_{x\in X}N'\longrightarrow \oplus_{x\in X}N$ is injective by Lemma \ref{L}.
\end{proof}

\begin{lemma}{\label{O}}
	Every projective left $(R,P_{\Omega})$-modules is a direct summand of a free left $(R,P_{\Omega})$-module.
\end{lemma}

\begin{proof}
	If $(M,\mathfrak{m}_{\Omega})$ is a projective left $(R,P_{\Omega})$-module and $(F(M), \overline{\mathfrak{m}}_{\Omega})$ is a free left $(R,P_{\Omega})$-module generated by set $M$, then there exists a left $(R,P_{\Omega})$-module homomorphism $\varphi:(F(M), \overline{\mathfrak{m}}_{\Omega})\longrightarrow (M,\mathfrak{m}_{\Omega})$ was induced by the identity set map $id_M: (M,\mathfrak{m}_{\Omega})\longrightarrow (M,\mathfrak{m}_{\Omega})$. By the universal property of $(F(M), \overline{\mathfrak{m}}_{\Omega})$, $\varphi$ is surjective with $\varphi|_M=id_M$. Since $(M,\mathfrak{m}_{\Omega})$ is a projective left $(R,P_{\Omega})$-module, there exists a left $(R,P_{\Omega})$-module homomorphism $\psi: (M,\mathfrak{m}_{\Omega})\longrightarrow (F(M),\overline{\mathfrak{m}}_{\Omega})$ such that $\varphi\circ \psi = id_{(M,\mathfrak{m}_{\Omega})}$. Then $\psi$ is injective, we have
	\begin{equation*}
		F(M)=\text{im}\psi\oplus\text{ker}\varphi\cong M\oplus\text{ker}\varphi.
	\end{equation*}
	Hence $(M,\mathfrak{m}_{\Omega})$ is the direct summand of $(F(M), \overline{\mathfrak{m}}_{\Omega})$.
\end{proof}

By Proposition \ref{M}, Theorem \ref{N} and Lemma \ref{O}, we obtain the following conclusion.

\begin{theorem}{\label{P}}
	Every projective left $(R,P_{\Omega})$-module is flat.
\end{theorem}
This result shows that the category of multiple Rota-Baxter modules contains sufficient flat objects, which enables the construction of tensor functors in this category.

\vspace{0.5cm}
{\bf Acknowledgments.} This research is supported by the National Natural Science Foundation of China (Grant No. 12301025 and 12101316) and the Belt and Road Innovative Talents Exchange Foreign Experts project (Grant No. DL2023014002L).

{\bf Declaration of interests.} The authors have no conflicts of interest to disclose.

{\bf Data availability.} No.

\end{document}